\documentclass{amsart}

\usepackage{amsmath, amscd, amsthm,amssymb, bbold}
\usepackage{
overpic, graphicx, tikz-cd}
\usepackage [autostyle, english = american]{csquotes}
\MakeOuterQuote{"}

\theoremstyle{plain}
\newtheorem{thm}{Theorem}[section]
\newtheorem{prop}[thm]{Proposition}
\newtheorem{lem}[thm]{Lemma}
\newtheorem{cor}[thm]{Corollary}
  
\theoremstyle{remark}
\newtheorem{rem}[thm]{Remark}
  
\theoremstyle{definition}
\newtheorem{defn}[thm]{Definition}
\newtheorem*{condition*}{Condition}

\newcommand{\real}{\mathbb{R}}
\newcommand{\integer}{\mathbb{Z}}
\newcommand{\complex}{\mathbb{C}}
\newcommand{\cone}{\mathbf{C}}
\newcommand{\mult}{\mathcal{M}}
\newcommand{\cL}{\mathcal{L}}
\newcommand{\cB}{\mathcal{B}}
\newcommand{\bL}{\mathbf{L}}
\newcommand{\bbL}{\mathbb{L}}
\newcommand{\back}{B}
\newcommand{\Fourier}{\mathcal{F}}
\newcommand{\Func}{\mathfrak{F}}
\newcommand{\smoothing}{\mathcal{K}}
\newcommand{\bi}{\mathbf{i}}
\newcommand{\bj}{\mathbf{j}}
\newcommand{\bm}{\mathbf{m}}
\newcommand{\bw}{\mathbf{w}}
\newcommand{\bx}{\mathbf{x}}
\newcommand{\by}{\mathbf{y}}
\newcommand{\bs}{\mathbf{s}}
\newcommand{\supp}{\mathrm{supp}\,}
\newcommand{\Tr}{\mathrm{Tr}\,}
\newcommand{\trace}{\mathrm{trace}}
\newcommand{\tracefree}{\mathrm{trace-free}}
\newcommand{\ess}{\mathrm{ess}}
\renewcommand{\top}{\mathrm{top}}
\newcommand{\excep}{\mathcal{E}}
\newcommand{\hn}{k}

\begin{document}


\title[The error term of the Prime Orbit Theorem]{The error term of The Prime Orbit Theorem\\
 for expanding semiflows}
\author{Masato Tsujii}
\date{\today}
\address{Department of Mathematics, Kyushu University, Motooka 744, Nishi-ku,
Fukuoka, 819-0395, Japan}
\email{tsujii@math.kyushu-u.ac.jp}

\begin{abstract}
We consider suspension semiflows of angle multiplying maps on the circle and study the distributions of periods of their periodic orbits. Under generic conditions on the roof function, we give an asymptotic formula on the number $\pi(T)$ of prime periodic orbits with period $\le T$. The error term is bounded, at least,  by 
\[
\exp\left(\left(1-\frac{1}{4\lceil \chi_{\max}/h_{\top}\rceil}+\varepsilon\right) h_{\top}\cdot  T\right)\qquad \mbox{in the limit $T\to \infty$}
\]
for arbitrarily small $\varepsilon>0$,  where $h_{\top}$ and $\chi_{\max}$ are respectively the topological entropy and the  maximal Lyapunov exponent of the semiflow.
\end{abstract}

\maketitle

\section{Introduction }

For a flow $f^{t}:M\to M$ on a closed manifold $M$ with some hyperbolicity, it is well known that the number $\pi(T)$ of periodic orbits with period $\le T$ grows exponentially as $T\to\infty$ and the exponential rate coincides with the topological entropy $h_{\top}$ of the flow. The prime orbit theorem, due to Parry and Pollicott \cite[Theorem 9.3]{PP}, gives a more precise estimate in the case of topologically weakly mixing hyperbolic flows: 
\begin{equation}\label{eq:PP}
\pi(T)=(1+o(1))\int_{1}^{T}\frac{e^{h_{\top}t}}{t}dt\quad \mbox{as $T\to\infty$.}
\end{equation}
This paper addresses estimates of the error term in this asymptotic formula. 

For geodesic flows on surfaces with negative (variable) curvature, Pollicott and Sharp proved in \cite{PS}  that the relative error term, denoted by $o(1)$ in the formula (\ref{eq:PP}) above, is actually exponentially small, that is, bounded by $Ce^{-\varepsilon T}$ with some $C>0$ and $\varepsilon>0$. More recently, this result is extended to higher dimensional cases by Giulietti, Liverani and Pollicott\cite{GLP} and  Stoyanov\cite{Stoyanov}. 
But not much is known about the exponential rate at which the relative error term decreases.  

For the geodesic flows on surfaces with negative {\em constant} curvature, we have a much more precise asymptotic formula due to Huber, which reads
\begin{equation}
\pi(T)=\int_{1}^{T}\frac{e^{h_{\top}t}}{t}dt+\sum_{i=1}^{k}\int_{1}^{T}\frac{e^{\mu_{i}t}}{t}dt+\mathcal{O}\left(e^{\rho t}\right)\label{eq:Huber}
\end{equation}
where $\rho=(3/4)h_{\top}$ and $\mu_{i}$, $1\le i\le k$,
are real numbers satisfying $\rho<\mu_{i}<h_{\top}$. 
(The exponents $\mu_{i}$ are related to small eigenvalues of the Laplacian on the surface. See \cite{Buser}.) But this result is known only for the case of constant curvature because the proof is based on the fact that the geodesic flow in such case is identified with the left action of a hyperbolic one-parameter subgroup of $SL(2,\real)$ on its  quotient space by the right action of a discrete subgroup.

Comparing these results, we are tempted to pose a question whether such a precise asymptotic formula as (\ref{eq:Huber}) is available for more general type of  hyperbolic flows and by a more flexible method. 
In this paper, we pursue  this question in the case of suspension semiflows of angle multiplying maps on the circle and provide a positive answer under generic conditions on the roof function.

\section{The main results}

\subsection{Definitions}

We consider a class of (simplest possible) expanding semiflows. This kind of semiflows have been studied in \cite{Ruelle86, Pollicott99, Tsujii08} as  a simplified model of Anosov flows. 
First we fix a positive integer $\ell\ge2$ and consider the angle-multiplying map
\[
\tau:S^{1}\to S^{1},\quad\tau(x)=\ell x\;\mod\integer.
\]
Let $C_{+}^{\infty}(S^{1})$ be the space of positive-valued $C^{\infty}$
functions on $S^{1}$.
Then we consider the suspension semiflow of $\tau$ with roof function~$f\in C_{+}^{\infty}(S^{1})$:
\[
\mathbf{T}_{f}=\{T_{f}^{t}:X_{f}\to X_{f}\mid t\ge0\}.
\]
(See Figure \ref{fig:expandingsemiflow}.)
This is a semiflow on the set 
\[
X_{f}:=\{(x,y)\in S^{1}\times\real\mid0\le y<f(x)\}\subset S^{1}\times\real
\]
and defined  precisely by the expression
\[
T_{f}^{t}(x,y)=(\tau^{n(x,y+t;f)}(x),\; y+t-f^{(n(x,y+t;f))}(x))
\]
where 
\begin{align}
 & f^{(n)}(x)=\sum_{i=0}^{n-1}f(\tau^{i}(x))\label{eq:fn}
 \intertext{and} 
 & n(x,t;f)=\max\{n\ge0\mid f^{(n)}(x)\le t\}.\label{eq:n}
\end{align}
\begin{figure}[htbp]
\begin{center}
\begin{overpic}[scale=0.35]{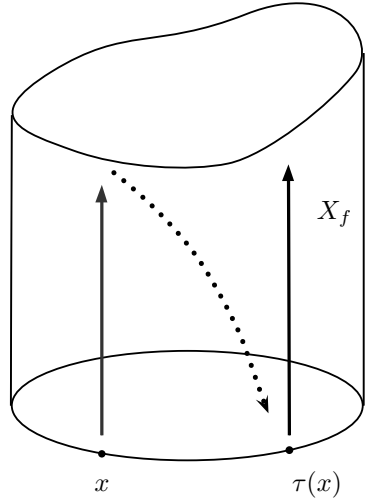}
\put(70,50){$X_f$}
\put(25,-5){$x$}
\put(65,-5){$\tau(x)$}
\end{overpic}
\end{center}
\caption{Expanding semiflow $\mathbf{T}_f$}
\label{fig:expandingsemiflow}
\end{figure}

\subsection{Spectral properties of transfer operators}\label{ss:heu}
By a heuristic argument, 
the distribution of periods of periodic orbits of $\mathbf{T}_f$ is related  to the spectra of the transfer operators 
\[
\cL^{t}\varphi(z)=\sum_{w:T_{f}^{t}(w)=z}\varphi(w).
\]
Indeed, computing
the flat trace of $\cL^{t}$, defined as the integral of the Schwartz kernel $K^{t}(z,w)$ of $\cL^t$ along
the diagonal $z=w$, we find 
\begin{equation}\label{eq:ABT}
\Tr^\flat \cL^t=\sum_{\gamma\in\Gamma}\sum_{n=1}^{\infty}\frac{|\gamma|}{1-E_{\gamma}^{-n}}\cdot\delta(t-n|\gamma|)
\end{equation}
where $\Gamma$ is the set of prime periodic orbits 
and   $|\gamma|$ and $E_{\gamma}$ denote respectively the prime period  
and the (coefficient of) linearized Poincar\'e map.
If we ignore the sum over $n\ge2$ and also the term $E_{\gamma}^{-n}$ in the denominator of the summands 
(which are in fact relatively small), we would have 
\begin{equation}\label{eq:principal}
\frac{1}{t}\cdot\Tr^\flat \cL^{t}\;\sim\;\sum_{\gamma\in\Gamma}\delta(t-|\gamma|),\qquad\mbox{ and so}\qquad\int_{1}^{T}\frac{1}{t}\cdot\Tr^\flat \cL^{t}dt\;\sim\;\pi(T).
\end{equation}
Therefore, if the flat trace $\Tr^\flat \cL^{t}$ were related to the spectrum
of $\cL^{t}$ as in the case of the usual trace,  the asymptotics of $\pi(T)$ would be expressed in terms of the spectrum of $\cL^{t}$. For this reason, we are going to study the spectral properties of the transfer operators $\cL^t$.

Let us say that a function $\varphi:X_{f}\to\complex$ is of class $C^{\infty}$
if $\cL^{t}\varphi$ for $t\ge 0$ are $C^{\infty}$ functions on the interior $X_{f}^{\circ}$ of $X_f$ (as a subset of $S^1\times \real$) and each of their partial derivatives are bounded. This implies that $\varphi$ is not only $C^{\infty}$ on $X_{f}^{\circ}$ but also satisfies some conditions on its values and differentials on the boundaries. Let $C^{\infty}(X_{f})$
be the space of $C^{\infty}$ functions on $X_{f}$ (defined as above), which,  equipped with the uniform $C^{r}$ norms $\|\varphi|_{X_{f}^{\circ}}\|_{C^{r}}$
for $r\ge0$,  is a Fr\'echet space. 
With this definition of $C^{\infty}(X_{f})$, we may regard $\cL^{t}$ for $t\ge 0$ as a continuous
semigroup of operators
\[
\cL^{t}:C^{\infty}(X_{f})\to C^{\infty}(X_{f}).
\]
To study spectral properties of $\cL^{t}$, we will define  Banach spaces 
\[
C^{\infty}(X_{f})\subset\cB^{r,p}(X_{f})\subset L^{2}(X_{f})
\]
for real numbers $r>0$
and integers $p\ge1$ and consider the natural extensions of $\cL^t$ to them.
The next theorem gives a spectral property of $\cL^{t}$ on  $\cB^{r,p}(X_{f})$
under some generic conditions on the roof function~$f$. 
We write $h(f)$, $\chi_{\max}(f)$ and $\chi_{\min}(f)$ respectively for the topological entropy, 
the maximum Lyapunov exponent and the minimum Lyapunov exponent:
\[
\chi_{\max}(f):=\lim_{t\to\infty}\frac{1}{t}\max_{z\in X_{f}}\log\|DT_{f}^{t}(z)\|,\qquad
\chi_{\min}(f):=\lim_{t\to\infty}\frac{1}{t}\min_{z\in X_{f}}\log\|DT_{f}^{t}(z)\|.
\]
We put
\[
\alpha(f):=\frac{\chi_{\max}(f)}{h(f)}.
\]
We always have $\alpha(f)\ge1$ 
from Ruelle inequality\cite{Mane} and may regard $\alpha(f)$ as a measurement of spacial
non-uniformity of expansion by the semiflow $\mathbf{T}_{f}$. 
\begin{thm}
\label{th:spec} For any $f\in C_{+}^{\infty}(S^{1})$, any $r>0$ and any integer $p\ge 1$,
the transfer operators $\cL^{t}$ for sufficiently large
$t>0$ extend to bounded operators
\begin{equation}\label{eq:cL}
\cL^{t}:\cB^{r,p}(X_{f})\to\cB^{r,p}(X_{f}).
\end{equation}
For each integer $p\ge1$ and for each $\varepsilon>0$, there is an open dense subset\/ $\mathcal{U}_{p}(\varepsilon)\subset C_{+}^{\infty}(S^{1})$ such that, if $f\in\mathcal{U}_{p}(\varepsilon)$ and if $r>0$ is sufficiently large (see (\ref{eq:choice_r})),
the essential spectral radius of the operator (\ref{eq:cL}) for 
sufficiently large $t>0$ is smaller than $\exp((\rho_p(f)+\varepsilon) t)$
where
\begin{equation}\label{eq:bound_ess}
\rho_{p}(f):=\frac{1}{2}\left(1+\frac{\max\left\{ p,\,\alpha(f)\right\} -1}{p}\right)\cdot h(f). 
\end{equation}
\end{thm}
\begin{rem}
\label{rem:mui} The conclusion of the theorem above implies that the spectral set of (\ref{eq:cL})  on the region $|z|\ge\exp((\rho_p(f)+\varepsilon) t)$
consists of finitely many eigenvalues with finite multiplicities. Such
eigenvalues (counted with multiplicity) are written in the form $\exp(\mu_{i}t)$,
$i=1,2,\cdots,k$, with complex numbers $\mu_{i}$  that do not depend
on $t$. (See \cite[pp295]{Tsujii08}.)
\end{rem}
 
The case  $p=1$ in the theorem above corresponds to the result in our previous paper \cite{Tsujii08}, where the bound is
\[
\rho_{1}(f)=\exp(\chi_{\max}(f)\cdot t/2) 
\]
as $\alpha(f)\ge 1$. (See also \cite{Tsujii10,Tsujii12} for the corresponding results for contact Anosov
flows.) This bound is preferable
when $\alpha(f)$ is close to $1$, but the claim becomes vacuous when $\alpha(f)\ge2$ for $\rho_1(f)$ exceeds the topological entropy $h_{\top}(f)$. The improvement achieved in Theorem \ref{th:spec} is that we get better bounds by choosing different integers $p\ge1$ depending on $\alpha(f)\ge1$. For simplicity's sake, suppose that  $f$ belongs to the residual subset $\mathcal{U}:=\cap_{p\in \mathbb{N}}\cap_{m=1}^\infty\,\mathcal{U}_{p}(1/m)\subset C_{+}^{\infty}(S^{1})$
and put 
\[
\rho(f):=\min_{p\ge1}\rho_{p}(f).
\]
Letting 
\footnote{This choice of $p$ is not always optimal.%
} $p(f)=\lceil\alpha(f)\rceil\ge1$, we have
\begin{equation}\label{eq:rhobound}
\rho(f)\le\rho_{p(f)}(f)  \le \left(1-\frac{1}{2p(f)}\right)h(f)<h(f).
\end{equation}
Therefore, by choosing suitable $p\ge 1$, we always get a bound for the
essential spectral radius of $\cL^t$ that is strictly smaller than the spectral
radius $\exp(h(f)t)$.

\subsection{Asymptotics of the number of periodic orbits}

We next give a consequence of Theorem~\ref{th:spec} on the remainder
term of the prime orbit theorem. Let $\Gamma=\Gamma(f)$ be the set
of prime periodic orbits for the semiflow $\mathbf{T}_{f}$. For
a prime periodic orbit $\gamma\in\Gamma$, we denote its period
by $|\gamma|$. Let $
\pi(T)=\#\{\gamma\in\Gamma\mid|\gamma|\le T\}$.

\begin{thm}\label{th:generalized_huber}
Let $\varepsilon>0$ and suppose that the roof function $f\in C_{+}^{\infty}(S^1)$ belongs
to the open dense subset $\mathcal{U}_{p}(\varepsilon)\subset C_{+}^{\infty}(X_{f})$
given in Theorem \ref{th:spec} for $p\ge 1$. Then, with setting 
\begin{equation}\label{eq:bar_rho}
\bar{\rho}=\bar{\rho}_p(f):=\frac{\rho_p(f)+h(f)}{2},
\end{equation}
we have an asymptotic formula 
\begin{align*}
\pi(T)=\int_{1}^{T}\frac{e^{h_{\top}t}}{t}dt+\sum_{i=1}^{k'}\int_{1}^{T}\frac{e^{\mu_{i}t}}{t}dt+\mathcal{O}\left(e^{(\bar{\rho}+\varepsilon) t}\right),
\end{align*}
where $\mu_{i}$, $1\le i\le k'$, are those complex numbers in Remark \ref{rem:mui}
satisfying the condition $\bar{\rho}+\varepsilon<\Re(\mu_{i})<h(f)$. 
\end{thm}
Note that, 
if we let $p=p(f)=\lceil\alpha(f)\rceil\ge1$, 
we have, from (\ref{eq:rhobound}), that
\[
\bar{\rho}_{p}(f)\le 
\left(1-\frac{1}{4\lceil \chi_{\max}(f)/h(f)\rceil}\right)h(f)<h(f).
\] 
\begin{rem}
The reason that we have the average $\bar{\rho}_p(f)$ instead of $\rho_p(f)$ in the statement of Theorem \ref{th:generalized_huber} above will appear in its proof given in Section \ref{sec:pf_th_huber}. (See Remark \ref{rem_bar_rho}.) 
\end{rem}

In the following sections, we proceed as follows. In Section \ref{sec:generic}, we formulate a transversality condition on the roof function $f$ and decompose Theorem \ref{th:spec} into two theorems:  Theorem \ref{thm:multiplicity} that proves prevalence of the transversality condition and Theorem \ref{thm:spectrum} that proves the conclusion of Theorem \ref{th:spec} from the transversality condition. We prove Theorem \ref{thm:spectrum} in Section \ref{sec:pf1} after preparation in Section \ref{sec:defcBonR}. 
We insert the proof of Theorem \ref{th:generalized_huber} in Section \ref{sec:pf_th_huber} before we prove Theorem~\ref{thm:multiplicity} in Section \ref{sec:multi}, as it uses the argument in the proof of Theorem \ref{thm:spectrum}.

\section{The generic condition}\label{sec:generic}

We set up notation on the dynamics of the semiflow $\mathbf{T}_{f}$
and formulate the transversality condition that defines the open dense
subset $\mathcal{U}_{p}(\varepsilon)$ in Theorem \ref{th:spec}.

\subsection{Differential of the semiflow $\mathbf{T}_{f}$}\label{ss:dif}

The differential $DT_{f}^{t}(z):\real^{2}\to\real^{2}$ at $z\in X_{f}$
is well-defined if $z$ and $T_{f}^{t}(z)$ are not on the (lower) boundary
of $X_{f}$. In general, we define 
\[
DT_{f}^{t}(z)=\lim_{\varepsilon\to+0}DT_{f}^{t}(x,y+\varepsilon):\real^{2}\to\real^{2},
\]
where $DT_{f}^{t}(x,y+\varepsilon)$ for sufficiently small $\varepsilon>0$
is constant and hence the limit on the right hand side is well-defined.
For $t\ge0$, we set 
\begin{equation}\label{eq:EF}
E(z,t;f)=\ell^{n(x,y+t;f)}\quad\mbox{and}\quad F(z,t;f)=\frac{\partial}{\partial x}f^{(n(x,y+t;f))}(x)
\end{equation}
where $n(x,t;f)$
and $f^{(n)}(x)$ are those defined in (\ref{eq:fn}) and (\ref{eq:n}). Then we have
\begin{equation}
DT_{f}^{t}(z)=\begin{pmatrix}E(z,t;f) & 0\\
F(z,t;f) & 1
\end{pmatrix}\label{eq:DT}
\end{equation}
We write $D^{\dag}T_{f}^{t}(z)$ for the transpose
of the inverse of $Df^{t}(z)$, that is,  
\begin{equation}
D^{\dag}T_{f}^{t}(z):={}^{T}(Df^{t}(z))^{-1}=\begin{pmatrix}E(z,t;f)^{-1} & S(z,t;f)\\
0 & 1
\end{pmatrix}\label{eq:DdagT}
\end{equation}
where 
\begin{equation}
S(z,t;f)=-E(z,t;f)^{-1}F(z,t;f).\label{eq:S}
\end{equation}
Then the minimum and maximum  Lyapunov exponents of $\mathbf{T}_{f}$ are written 
\begin{align*}
\chi_{\min}(f)&=\lim_{t\to\infty}\frac{1}{t}\log\left(\min_{z\in X_{f}}E(z,t;f)\right)
\intertext{and}
\chi_{\max}(f)&=\lim_{t\to\infty}\frac{1}{t}\log\left(\max_{z\in X_{f}}E(z,t;f)\right).
\end{align*}
For the topological entropy $h(f)$, we have 
\[
\frac{1}{t}\log\left(\min_{z\in X_{f}}E(z,t;f)\right)\le h(f)\le\frac{1}{t}\log\left(\max_{z\in X_{f}}E(z,t;f)\right)
\]
for any $t>0$ and hence 
\[
\chi_{\min}(f)\le h(f)\le\chi_{\max}(f).
\]

For $0<y_{\min}<y_{\max}$ and $\kappa_{0}>0$,
let $
\Func(y_{\min},y_{\max},\kappa_{0})\subset C_{+}^{\infty}(S^{1})$
be the open subset that consists of  $f\in C_{+}^{\infty}(S^{1})$ satisfying 
\[
y_{\min}<f(x)<y_{\max},\quad |f'(x)|< \kappa_0\quad ,|f''(x)|<\kappa_{0}\quad\mbox{for all \ensuremath{x\in S^{1}}.}
\]
If $f\in\Func(y_{\min},y_{\max},\kappa_{0})$, we have 
\begin{equation}
\bar{\chi}_{\min}:=\frac{\log\ell}{y_{\max}}\le\chi_{\min}(f)\le h(f)\le \chi_{\max}(f)\le \bar{\chi}_{\max}:=\frac{\log\ell}{y_{\min}}.\label{eq:bound_Lyapunov}
\end{equation}
In what follows, we fix  $0<y_{\min}<y_{\max}$ and $\kappa_{0}>0$ and confine our attention to the semiflows $\mathbf{T}_{f}$
with $f\in\Func(y_{\min},y_{\max},\kappa_{0})$. Since the subset
$\Func(y_{\min},y_{\max},\kappa_{0})$ exhausts $C_{+}^{\infty}(S^{1})$
in the limit $y_{\min}\to+0$, $y_{\max}\to+\infty$ and $\kappa_{0}\to+\infty$,
this causes no loss of generality.
We henceforth fix $r>0$ such that  
\begin{equation}\label{eq:choice_r}
r> \bar{\chi}_{\max}/\bar{\chi}_{\min}\ge 1.
\end{equation}

\subsection{Cones in the flow direction}
Since the time-$t$-map $T^t_f$  is partially hyperbolic, its (push-forward) action  on the cotangent
bundle
\[
D^{\dag}T_{f}^{t}:X_{f}\times\real^{2}\to X_{f}\times\real^{2},\quad D^{\dag}T_{f}^{t}(z,\xi)=(T_{f}^{t}(z),D^{\dag}T_{f}^{t}(z)\xi)
\]
admits a forward invariant cone field. We can set up such a cone field concretely as follows.
For real numbers $s$ and $\theta>0$, we define  
\[
\cone(s,\theta):=\{(\xi,\eta)\in\real^{2}\mid|\xi-s\eta|\le\theta|\eta|\}\subset\real^{2}.
\]
We fix a real number  $\gamma_{0}$ satisfying $1/\ell<\gamma_{0}<1$
and set
\[
\cone_{0}:=\cone(0,\theta_{0}):=\{(\xi,\eta)\in\real^{2}\mid|\xi|\le\theta_{0}|\eta|\}\subset\real^{2}
\]
where 
\[
\theta_{0}:=\frac{\kappa_{0}}{\gamma_{0}\ell-1}.
\]
Then we have that
\begin{equation}
(DT_{f}^{t})_{z}^{\dag}(\cone_{0})=\cone(S(z,t;f),E(z,t;f)^{-1}\theta_{0})\subset\cone(0,\gamma_{0}\theta_{0})\subset\cone_{0}\label{eq:DdagTcone}
\end{equation}
for all $z=(x,y)\in X_{f}$ and $t\ge f(x)-y$.

\subsection{Backward orbits}\label{ss:bo}

For each $z\in X_{f}$, the number of points in its backward orbit
\[
(T_{f}^{t})^{-1}(z)=\{w\in X_{f}\mid T_{f}^{t}(w)=z\}
\]
for time $t>0$ grows exponentially  as $t\to0$. Indeed, for any
$\varepsilon>0$, there exists $C_{\varepsilon}>1$ such that 
\begin{equation}\label{eq:card_backward_orbit}
C^{-1}_{\varepsilon}e^{(h(f)-\varepsilon)t}<\#(T_{f}^{t})^{-1}(z)<C_{\varepsilon}e^{(h(f)+\varepsilon)t}\quad\forall z\in X_{f}, \quad \forall t\ge0.
\end{equation}
For $z=(x,y)\in X_{f}$, $t\ge0$ and $w\in(T_{f}^{t})^{-1}(z)$,
let 
\begin{equation}\label{eq:sn}
0< s_{\hn(z,w;t)}(z,w;t)<\cdots<s_{2}(z,w;t)<s_{1}(z,w;t)\le t
\end{equation}
be the sequence of time $t$ at which the orbit $T_{f}^{s}(w)$, $0< s\le t$, crosses the lower boundary $S^{1}\times\{0\}$ of~$X_{f}$. By definition,
we have 
\[
 T_{f}^{s_{k}(z,w;t)}(w)\in\tau^{-k}(x)\times\{0\}\quad\mbox{for $1\le k\le \hn(z,w;t)$.}
\]
Since we are assuming that $f\in\Func(y_{\min},y_{\max},\kappa_{0})$, we have 
\[
\lfloor t/y_{\max}\rfloor\le n(z,w;t)\le\lceil t/y_{\min}\rceil.
\]

Below we investigate transversality between the cones
\begin{equation}
(D^{\dag}T_{f}^{t})_{w}(\cone_{0})=\cone(S(w,t;f),E(w,t;f)^{-1}\theta_{0})\quad\mbox{for } w\in(T_{f}^{t})^{-1}(z)\label{eq:imageOfCones}
\end{equation}
in some generalized sense. Since much variety of angles of the cones $(D^{\dag}T_{f}^{t})_{w}(\cone_{0})$ for $w\in (T_{f}^{t})^{-1}(z)$ causes  technical difficulties in the following argument, 
 we classify the points $w\in (T_{f}^{t})^{-1}(z)$ with respect
to the value of $E(w,t;f)$ (whose reciprocal is proportional to the angle of  $(D^{\dag}T_{f}^{t})_{w}(\cone_{0})$). For an interval
$J=[a,b]$ with $0<a<b$, we set 
\[
\back(z,t;J;f)=\{w\in(T_{f}^{t})^{-1}(z)\mid e^{at}\le E(w,t;f)\le e^{bt}\}.
\]
We fix a  $C^{\infty}$ function $\chi:\real\to[0,1]$  such
that 
\begin{equation}\label{eq:chi}
\chi(t)=\begin{cases}
0, & \mbox{ if \ensuremath{t\ge2};}\\
1, & \mbox{ if \ensuremath{t\le1}.}
\end{cases}
\end{equation}
For $s\in \real$, let $
\langle s\rangle=\chi(s)+(1-\chi(s))|s|$,
so that $\langle s\rangle \in  [1,\max\{1,|s|\}]$ and that
\[
\langle s\rangle=
\begin{cases}
1, &\mbox{ if $|s|\le 1$;}\\
|s|, &\mbox{ if $|s|\ge 2$}. 
\end{cases}
\]
\begin{defn}
For $z\in X_f$, $t>0$ and  a $p$-tuple $\bw=(\bw(1),\cdots,\bw(p))$
of points in $(T_{f}^{t})^{-1}(z)$, we set 
\[
S(\bw,t;f)=\sum_{i=1}^{p}S(\bw(i),t;f)
\]
and define $E(\bw,t;f)$ by the relation
\[
\frac{1}{E(\bw,t;f)}=\sum_{i=1}^{p}\frac{1}{E(\bw(i),t;f)} .
\]
We define the function $W^{r}(\bw,t;f):\real^{2}\to\real^{2}$
by 
\[
W^{r}(\bw,t;f)(\xi,\eta)=\left\langle \frac{E(\bw,t;f)\cdot |\xi-S(\bw,t;f)\eta|}{\theta_{0}\cdot \langle\eta\rangle}\right\rangle ^{r}.
\]
This function takes constant value
$1$ on the cone 
\begin{equation}\label{eq:coneSE}
\cone(S(\bw,t;f),E(\bw,t;f)^{-1}\theta_0)
\end{equation}
and grows rapidly as the point gets far from it. 
\end{defn}
As a quantification of transversality between
p-tuples of cones in (\ref{eq:imageOfCones}),
we consider the quantity 
\begin{equation}\label{eq:mult}
\sup_{\xi\in \real} \left(
\sum_{\bw=(w(1),\cdots,w(p))\in\back(z,t;J;f)^{p}}\frac{1}{W^{r}(\bw,t;f)(\xi,2)}\right).
\end{equation}
In the case $p=1$, boundedness of this quantity by some relatively small constant implies that most of the cones in (\ref{eq:imageOfCones}) are transversal to each other. 
In the case $p>1$, it does not have such geometric meaning, but still useful in the argument below.   

\begin{rem}
By definition, $W^{r}(\bw,t;f)(\xi,\eta)$ is constant on the intersection of a straight line through the origin with the region $|\eta|\ge 2$.
Hence the constant $2$ in $W^{r}(\bw,t;f)(\xi,2)$ above could be any constant $\eta_0$ such that $|\eta_0|\ge 2$.
\end{rem}
The next theorem gives a bound on (a slight modification of) the quantity (\ref{eq:mult})
under generic conditions on the roof function $f$. Before stating the theorem, let us make a guess on the bound. Note that each function
$\xi\mapsto W^{r}(\bw,t;f)(\xi,1)^{-1}$ decays rapidly on the outside of a neighborhood of $\xi=S(\bw,t;f)$ with width proportional to $E(\bw,t;f)^{-1}\le e^{-at}$.
Hence, if the values of $S(\bw,t;f)$ for $\bw\in\back(z,t;J;f)^{p}$
were distributed randomly and independently on the interval $[-p\theta_{0},p\theta_{0}]$
(as random variables on the space of roof functions $f$), the large deviation argument would tell that, for almost all roof functions $f$,  the quantity (\ref{eq:mult}) should be bounded by 
\[
e^{\varepsilon t} \max\{1,\exp(-at)\cdot(\sharp(T_{f}^{t})^{-1}(z))^{p}\}
\le \exp\left(\left(\max\{p\cdot h(f)-a,0\}+\varepsilon\right)t\right)
\]
in the limit $t\to \infty$, for arbitrarily small $\varepsilon>0$. The next theorem tells that this guess is basically true, but with slight modifications. 

For an integer $n\ge 1$, let $\mathrm{Per}(\tau,n)$ be the set of periodic points of $\tau$ with period not greater than $n$ and, for $\delta>0$, let $\mathrm{Per}_\delta(\tau,n)$ be the open $\delta$-neighborhood of $\mathrm{Per}(\tau,n)$. 
\begin{thm}
\label{thm:multiplicity} 
Let $p\ge 1$ be an integer. 
For an interval $J=[a,b]$ with $0<a<b$ and real numbers $\varepsilon,\delta>0$, there exists $n_0=n_0(\varepsilon)$ and  a prevalent\footnote{See remark below for the definition of this term "prevalent".} subset 
\[
\mathcal{G}(J,n,\varepsilon,\delta;p)\subset\Func(y_{\min},y_{\max},\kappa_{0})
\]
for $n\ge n_0$, 
such that the following claim holds for $f\in\mathcal{G}(J,n,\varepsilon,\delta;p)$:\newline 
\indent For sufficiently large $t>0$ and for any  $z=(x,y)\in X_{f}$ with $x\notin \mathrm{Per}_\delta(n,\tau)$, there exists  a subset $\excep=\excep(z,t;f)\subset\tau^{-n}(x)$ with $\#\excep\le p\lceil 10a/\varepsilon\rceil$ such that 
\begin{equation}\label{eq:mult2}
\sum{}^{*}\;\frac{1}{W^{r}(\bw,t;f)(\xi,2)}<\exp((\max\{p\cdot h(f)-a,0\}+p(b-a)+\varepsilon)t)
\end{equation}
for any $\xi\in [-\theta_0,\theta_0]$, 
where the sum $\sum^{*}$ is taken over $p$-tuples $\bw=(\bw(1),\cdots\!,\bw(p))$ in $\back(z,t;J;f)^{p}$
with 
\[
T_{f}^{s_{n}(z,\bw(i);t)}(\bw(i))\notin \excep\times\{0\}\quad\mbox{for \ensuremath{i=1,2,\cdots,p}.}
\]
\end{thm}
Note that the point $T_{f}^{s_{n}(z,\bw(i);t)}(\bw(i))$ belongs to $\tau^{-n}(x)\times\{0\}$ from the definition of $s_{n}(z,w;t)$ in (\ref{eq:sn}).
\begin{rem}\label{rem:prevalence}
In the statement above, we used the notion of "prevalence" that is introduced in \cite{MR1161274, MR1191479, Yorke}. A measurable subset $S$ in a linear topological space $X$ is said to be \emph{shy} if there exists a Borel measure $\mu$ such that 
$0<\mu(U)<\infty$ for some compact subset $U\subset X$ and $\mu(S+x)=0$ for any $x\in X$. ($\mu$ is called a transverse measure for $S$.) A shy subset has empty interior. Any countable union of shy subsets is again shy. (This is far from trivial.) A measurable subset $P$ is said to be \emph{prevalent} in $Q\subset X$ if $Q\setminus P$ is shy. (See \cite{Yorke} for the detail. See also \cite{MR1167374} for a similar but different notion which could be used alternatively.)
\end{rem}

The next theorem states that the transversality condition in the theorem above yields an estimate on the essential spectral radius of the transfer operator $\cL^t$. 
\begin{thm}
\label{thm:spectrum} Let $p\ge 1$ be an integer and let  $J_{\nu}=[a_{\nu},b_{\nu}]$, $1\le \nu\le \nu_0$,
be intervals such that the union of their interiors contains the interval $[\bar{\chi}_{\min},\bar{\chi}_{\max}]$.
For $1\le \nu\le \nu_0$, put 
\begin{equation}\label{eq:boundEss} 
\mu_{\nu}=\frac{(p-1)h(f)+\max\{p\cdot h(f)-a_{\nu},0\}+p(b_\nu-a_\nu)+b_{\nu}}{2p}.
\end{equation} 
Suppose that $f_0$ belongs to the prevalent subset 
\[
\mathcal{G}=\bigcap_{\nu= 1}^{\nu_0}\bigcap_{m=1}^\infty \bigcap_{m'=1}^\infty 
\bigcap_{n\ge n_0(1/m)}\mathcal{G}(J_{\nu},n, 1/m, 1/m';p)\subset \Func(y_{\min},y_{\max},\kappa_{0})
\]
where $\mathcal{G}(J,n,\varepsilon,\delta; p)$ is that  in
Theorem \ref{thm:multiplicity}. Then, for any $\eta>0$, there exists a neighborhood $\mathcal{V}$ of $f_0$ in $C^\infty_+(S^1)$ such that, if  $f\in \mathcal{V}$,  the essential spectral radius
of the transfer operator (\ref{eq:cL})  for sufficiently large $t$ is bounded by $e^{(\mu(f)+\eta) t}$ where 
\begin{equation}\label{eq:muf}
\mu(f)=\max\{ \mu_{\nu}\mid \mathrm{int}\,J_{\nu}\cap [\chi_{\min}(f),\chi_{\max}(f)]\neq \emptyset \}.
\end{equation}
\end{thm}
For given $\eta>0$,
we can take the intervals $J_{\nu}=[a_{\nu},b_{\nu}]$, $1\le \nu\le \nu_0$,
narrow enough so that the quantity $\mu(f)$  is
bounded by 
\[
\frac{\left(p-1+\max\{p,\alpha(f)\}\right)h(f)+\eta}{2p}= \rho_p(f)+\frac{\eta}{2p}.
\]
Therefore  Theorem~\ref{th:spec} follows from Theorem~\ref{thm:multiplicity} and Theorem \ref{thm:spectrum}. (The first claim on boundedness of $\cL^t$ is proved in Subsection \ref{ss:defcB}.)

\section{The Banach space $\cB^{r,p}(\protect\real^{2})$}
\label{sec:defcBonR}
In this section,  we define the Banach space $\cB^{r,p}(\real^{2})$
and prove some related lemmas. (The definition resembles that of Besov spaces in \cite[Section 2.3.1]{Triebel}.)
We will construct
the Banach space $\cB^{r,p}(X_{f})$ in (\ref{eq:cL}) using  this Banach space as its local model. 
(But, since $X_f$ is not a manifold, the construction is a little different from the usual one.)

\subsection{Definitions}
We introduce two partitions of unity on
$\real$:
\[
\{\chi_{n}:\real\to[0,1]\}_{m\in\integer_{+}}\quad\mbox{and}\quad \{\rho_{n}:\real\to[0,1]\}_{n\in\integer}.
\]
The former is the Littlewood-Paley
partition of unity, defined by 
\[
\chi_{m}:\real\to[0,1],\quad\chi_{m}(t)=\begin{cases}
\chi(|t|), & \mbox{ if $m=0$;}\\
\chi(2^{-m}|t|)-\chi(2^{-m+1}|t|), & \mbox{ if $m\ge 1$}
\end{cases}
\]
where $\chi$ is that taken in Subsection \ref{ss:bo}. 
The latter  
is defined by 
\[
\rho_{n}=\begin{cases}
\chi(\mathrm{sgn}(x) \sqrt{|x|}-n+1)-\chi(\mathrm{sgn}(x)\sqrt{|x|}-n+2), & \mbox{if $n\ge  1$};\\
\chi(\sqrt{|x|}+1), & \mbox{if $n=0$;}\\
\chi(\mathrm{sgn}(x) \sqrt{|x|}+n+1)-\chi(\mathrm{sgn}(x)\sqrt{|x|}+n+2), & \mbox{if \ensuremath{n\le-1}}.
\end{cases}
\]
Note that the support of the function $\rho_n$ is contained in the interval 
\[
I_{n}=\begin{cases}
[(n-1)^{2},(n+1)^{2}], & \mbox{if $n\ge1$};\\
[-1,1], & \mbox{if $n=0$;}\\
[-(|n|+1)^{2},-(|n|-1)^{2}], & \mbox{if $n\le-1$}
\end{cases}
\]
which contains $\mathrm{sgn}(n)\cdot n^2$ and whose length is comparable with $|n|$. 

Next we define the partition of unity 
\[
\{\chi_{n,m}:\real^{2}\to[0,1]\mid n\in\integer,m\in\integer_{+}\}
\]
on $\real^{2}$ by 
\[
\chi_{n,m}:\real^{2}\to[0,1],\quad\chi_{n,m}(\xi,\eta)=\rho_{n}(\eta)\cdot\chi_{m}(\theta_{0}^{-1}\cdot\langle n\rangle^{-2}\cdot\xi).
\]
The support of the function $\chi_{n,m}$ is contained in the region 
\[
\left([-2^{m+1}\langle n\rangle^2\theta_0,-2^{m-1}\langle n\rangle^2\theta_0]\cup [2^{m-1}\langle n\rangle^2\theta_0, 2^{m+1}\langle n\rangle^2\theta_0]\right)\times I_n
\]
when $m\ge 1$, and  in $
[-2\langle n\rangle^2\theta_0, 2\langle n\rangle^2\theta_0])\times I_n
$ otherwise. 
\begin{defn}
For a real number $r> 0$ and an integer $p\ge1$, we define the
norm $\|\cdot\|_{r,p}$ on the Schwartz space $\mathcal{S}(\real^{2})$
by 
\begin{equation}\label{eq:norm_Brp}
\|u\|_{r,p}=\left(\sum_{n=-\infty}^{\infty}\sum_{m=0}^{\infty}(2^{rm}\cdot\|\Fourier^{-1}\circ \mult(\chi_{n,m})\circ \Fourier u\|_{L^{2p}})^{2p}\right)^{1/2p}
\end{equation}
where $\Fourier$ and $\mult(\varphi)$ denote the  Fourier transform and the multiplication operator by $\varphi$ respectively, and $\|\cdot\|_{L^{2p}}$ denotes the $L^{2p}$ norm. Let $\cB^{r,p}(\real^{2})\subset\mathcal{S}'(\real^{2})$
be the completion of $\mathcal{S}(\real^{2})$ with respect to this
norm. For a subset $K\subset\real^{2}$, we write $\cB^{r,p}(K)$
for the subspace of $\cB^{r,p}(\real^{2})$ that consists of elements
whose support is contained in the closure of $K$. 
\end{defn}
\begin{rem}\label{rem:brpq} We could introduce another parameter $q\in \real$ and define the Banach space 
$\cB^{r,p,q}(\real^{2})$ as the completion of $\mathcal{S}(\real^{2})$ with respect to the norm
\[
\|u\|_{r,p,q}=\left(\sum_{n=-\infty}^{\infty}\sum_{m=0}^{\infty}(2^{rm}\cdot\langle n\rangle^{2q}\cdot\|\Fourier^{-1}\circ \mult(\chi_{n,m})\circ \Fourier u)\|_{L^{2p}})^{2p}\right)^{1/2p}.
\]
We can develop our argument presented below for these more general Banach spaces (regardless of the choice of $q$) in parallel, with slight differences in constants. One advantage of considering such generalization is that we can prove that the eigenfunctions of $\cL^t$ corresponding to the peripheral eigenvalues outside of the essential spectral radius belong to $C^\infty(
X_f)$. This is essentially because $\cap_{r,q} \cB^{r,p,q}(\real^2)= C^\infty(\real^2)$ and because the peripheral eigenvalues and the corresponding eigenfunctions do not depend on the choice of Banach spaces. But we restrict our argument below to the case $q=0$ for simplicity's sake.\end{rem}

For technical argument in the next subsection, 
we introduce slight variants of the Banach space $\cB^{r,p}(\real^{2})$.
For real numbers $S$ and $E>0$, let $A_{S,E}:\real^{2}\to\real^{2}$
be the linear map defined by 
\begin{equation}\label{eq:ASE}
A_{S,E}\begin{pmatrix}x\\
y
\end{pmatrix}=\begin{pmatrix}Ex\\
SEx+y
\end{pmatrix}=\begin{pmatrix}E & 0\\
SE & 1
\end{pmatrix}\begin{pmatrix}x\\
y
\end{pmatrix}.
\end{equation}
The transpose of its inverse is  
\[
A_{S,E}^{\dagger}\begin{pmatrix}\xi\\
\eta
\end{pmatrix}=\begin{pmatrix}E^{-1} & -S\\
0 & 1
\end{pmatrix}\begin{pmatrix}\xi\\
\eta
\end{pmatrix}.
\]
The Banach space $\cB_{S,E}^{r,p}(\real^2)$ is defined 
as the push-forward of $\cB^{r,p}(\real^2)$ by $A_{S,E}$. Precisely we define 
\[
\cB^{r,p}(\real^2)=\{u\in\mathcal{D}'(\real^2)\mid u\circ A_{S,E}\in \cB^{r,p}(\real^2)\}
\]
and equip it with the norm 
\begin{align}\label{eq:rpse}
\|u\|_{r,p,S,E}&:=E^{1/2p}\cdot \|u\circ A_{S,E}\|_{r,p}\\
&=\left(\sum_{n=-\infty}^{\infty}\sum_{m=0}^{\infty}(2^{rm}\|\Fourier^{-1}\circ \mult(\chi_{n,m,S,E})\circ\Fourier u)\|_{L^{2p}})^{2p}\right)^{1/2p}\notag
\end{align}
where $\chi_{n,m,S,E}:=\chi_{n,m}\circ(A_{S,E}^{\dagger})^{-1}$. 

\subsection{Basic estimates}
We provide a few basic lemmas related to the definitions introduced above. Note  that the operator $\Fourier^{-1}\circ \mult(\chi_{n,m})\circ \Fourier$ is written as a convolution operator 
\[
\Fourier^{-1}\circ \mult(\chi_{n,m})\circ \Fourier u=\hat\chi_{n,m}*u
\]
with $
\hat\chi_{n,m}=(2\pi)^{-1}\Fourier^{-1} \chi_{n,m}$. 
\begin{lem}
\label{lem:Fxnm}For arbitrarily large $\nu>0$, there exists a constant
$C_{\nu}$ such that 
\[
|\hat\chi_{n,m}(x,y)|\le C_{\nu}\cdot(2^{m}\langle n\rangle^{3})\cdot\langle 2^{m}\langle n\rangle^{2}|x|\rangle^{-\nu}\cdot\langle\langle n\rangle\cdot|y|\rangle^{-\nu}
\]
uniformly for integers $n$ and $m\ge 0$. In particular, the $L^1$ norm of $\hat\chi_{n,m}$ is uniformly bounded. 
\end{lem}
\begin{proof}
The family of functions 
\[
X_{n,m}(\xi,\eta):=\chi_{n,m}(2^{m}\langle n\rangle^{2}\xi,\langle n\rangle(\eta-n|n|))
\]
for $n\in\integer$ and $m\in\integer_{+}$ are uniformly bounded in $\mathcal{S}(\real^{2})$ and therefore so are the family of functions 
\[
\Fourier^{-1}X_{n,m}(x,y)=(2^{-m}\langle n\rangle^{-3})\cdot
e^{i n|n| y}\cdot \Fourier^{-1}\chi_{n,m}(2^{-m}\langle n\rangle^{-2}x,\langle n\rangle^{-1}y).
\]
This implies the conclusion of the lemma.
\end{proof}
Similarly we have
\begin{lem}\label{lem:Fxnm2}
The $L^1$ norm of $\hat\chi_{n,m,S,E}=(2\pi)^{-1}\Fourier^{-1} \chi_{n,m,S,E}$ is bounded by a constant independent of $n$, $m$, $S$ and $E$.
\end{lem}

By abuse of notation, we will write $\hat{\chi}_{n,m}$ also for the convolution operator by $\hat{\chi}_{n,m}$, so that $\hat{\chi}_{n,m}u=\hat{\chi}_{n,m}*u=\Fourier^{-1}\circ \mult({\chi}_{n,m})\circ \Fourier u$.  
\begin{lem}\label{lem:xnm_trace}
For integers $n$ and $m\ge 0$ and for a bounded region $U\subset \real^2$, the convolution operator 
\[
\hat{\chi}_{n,m}=\Fourier^{-1}\circ \mult({\chi}_{n,m})\circ \Fourier:L^{2p}(U)\to L^{2p}(\real^2)
\]
is a trace class operator. There exists a constant $C_0>0$, independent of $n$, $m$ and $U$, such that
\[
\|\hat{\chi}_{n,m}:L^{2p}(U)\to L^{2p}(\real^2)\|_{\Tr}\le 
C_0 \cdot 2^m \langle n\rangle^3\cdot |U|_{n,m}
\]
where $\|\cdot\|_{\Tr}$ denotes the trace norm and 
\[
|U|_{n,m}:=\int \max_{(x,y)\in U}\left(
\langle 2^m \langle n\rangle^2 |x-x'|\rangle^{-2}\cdot 
\langle \langle n\rangle |y-y'|\rangle^{-2} \right)dx' dy'.
\]
\end{lem}
\begin{rem}
Since we consider operators between Banach spaces, it might be more standard to use the term "nuclear operator" and "nuclear norm" instead of "trace class operator" and "trace norm".
For the definition and basic properties of trace class (or nuclear) operators, we refer \cite[Ch.~5]{GGK}.
\end{rem}
\begin{proof}
Let us set $\chi'_{n,m}:=\sum^*_{n',m'} \chi_{n',m'}$ where the sum $\sum^*$ is taken over $(n',m')$ such that $\supp \chi_{n',m'}\cap \supp \chi_{n,m}\neq \emptyset$. Since $\chi'_{n,m}\cdot \chi_{n,m}=\chi_{n,m}$, we may write the operator $\hat{\chi}_{n,m}$ as 
\begin{align*}
\hat{\chi}_{n,m}u&= \hat\chi'_{n,m}* \hat\chi_{n,m}* u=\int  \phi_{z'} u \,dz'
\end{align*}
where $\phi_{z'}$ is the rank one operator defined by
\[
\phi_{z'} u(z)= \left(\int \hat \chi_{n,m}(z'-z'') u(z'') dz''\right)\cdot \hat\chi'_{n,m}(z-z').
\]
From Lemma \ref{lem:Fxnm}, we have
\begin{align*}
&\|\phi_{z'}:L^{2p}(U)\to L^{2p}(\real^2)\|_{\Tr}\\
&\qquad \le C_0 \langle n\rangle^{3} 2^{m}\cdot \max_{(x,y)\in U}\left(
\langle 2^m \langle n\rangle^2 |x-x'|\rangle^{-2}
\langle \langle n\rangle |y-y'|\rangle^{-2} \right)
\end{align*}
for $z'=(x',y')$. Indeed the left hand side is bounded by 
\[
\|\hat{\chi}_{n,m}\|_{L^{2p}}\cdot \|\hat\chi'_{n,m}\|_{(L^{2p})^*}= \|\chi_{n,m}\|_{L^{2p}}\cdot \|\hat\chi'_{n,m}\|_{L^q}
\]
with $q>0$ such that $q^{-1}+(2p)^{-1}=1$ and hence by $C_0 \langle n\rangle^{3} 2^{m}$ at least. Because $\nu>0$ in Lemma \ref{lem:Fxnm} is arbitrarily large, we can get the latter term on the right hand side in addition.  
Finally we obtain the lemma by the triangle inequality.  
\end{proof}

For the purpose of  extracting the low-frequency part of functions, we consider the operators  
\[
\smoothing_{k}:\mathcal{S}'(\real^{2})\to\mathcal{S}(\real^{2}),\quad\smoothing_{k}u=
\sum_{n,m: 2^m \langle n\rangle^2\le k}\hat{\chi}_{n,m}u
\]
for integers $k>0$. 
If  $U\subset\real^{2}$
is a bounded region, the operator  $\smoothing_{k}:\cB^{r,p}(U)\to \cB^{r,p}(\real^{2})$ is a trace class operator from Lemma \ref{lem:xnm_trace} and hence compact. 

As a model of  the semiflow $T_f^t$ viewed in local charts (that we will choose in the next section), we consider a  $C^{\infty}$ diffeomorphism 
\begin{equation}\label{eq:A}
A:V\to A(V)\subset\real^{2}, \quad A(x,y)=(Ex,y+g(Ex))
\end{equation}
where $E\ge 1$, $V:=(-E^{-1}\eta_{*},E^{-1}\eta_{*})\times\real\subset \real^2$ with some small $\eta_*>0$ and $g:(-\eta_{*},\eta_{*})\to\real$ is a $C^{\infty}$ function satisfying $|g'(x)|\le \gamma_0 \theta_0$. 
Letting $\varphi:\real^{2}\to\real$ 
be a $C^{\infty}$ function whose support is contained in $(-\eta_*,\eta_*)\times \real$, we consider the transfer operator 
\begin{equation}\label{eq:L_local}
L:C^{\infty}_0(V)\to C^{\infty}_0(A(V)),\quad Lu=(\varphi\cdot u)\circ A^{-1}.
\end{equation}
In the next proposition, we suppose that the function $\varphi(x,y)\in C^\infty_0(\real^2)$ satisfies
\begin{equation}\label{eq:py}
\left\|\frac{\partial^m \varphi}{\partial y^m}\right\|_{\infty}\le K_m\quad \mbox{for $m\ge 0$}
\end{equation}
for some given constants $K_m>0$. When we apply the proposition below in the next section, we will consider many different functions as $\varphi$, which satisfy the condition (\ref{eq:py}) for some uniform constants $K_m$.
\begin{prop}
\label{lm:local1}  
If we have (in addition to the setting above) that
\begin{equation}\label{eq:var_g}
|g'(x)-g'(0)|<(1-\gamma_{0})\theta_{0}/E\quad
\mbox{for all $x\in (-\eta_*,\eta_*)$},
\end{equation} 
the operator $L$ extends to a bounded operator
\[
L:\cB^{r,p}(V)\to \cB_{S,E}^{r,p}(A(\supp\varphi))\quad \mbox{where $S=g'(0)$.}
\] 
There exists a constant $C_0>0$, which depends only on $p$, $r$ and the constants $K_m$'s in $(\ref{eq:py})$, such that we have
\[
\|L\circ(1-\smoothing_{k}):\cB^{r,p}(V)\to \cB_{S,E}^{r,p}(A(\supp \varphi))\|\le C_{0} E^{1/2p}
\]
provided that  we take  sufficiently large $k>0$ according to $A$ and $\varphi$. 
\end{prop}
\begin{proof}
Recall the linear map $A_{S,E}$ in (\ref{eq:ASE}). The diffeomorphism $A_{S,E}^{-1}\circ A$ satisfies the assumption on $A$ for the case $E=1$ and $S=0$. Recall also that $\cB_{S,E}^{r,p}(\real^2)$ is defined as the push-forward of $\cB^{r,p}(\real^2)$ by $A_{S,E}$ and equipped with the norm (\ref{eq:rpse}) having the factor $E^{1/2p}$. Hence, to prove the statement of the lemma, it is enough to prove it in the case $E=1$ and $S=0$ (and $\cB_{S,E}^{r,p}(\real^2)=\cB^{r,p}(\real^2)$ consequently). 

We assume $E=1$ and $S=0$. Take $u\in\mathcal{S}(\real^{2})$ arbitrarily and set 
\begin{align*}
&u_{n,m}  =\hat{\chi}_{n,m}(u),\\
&v_{(n,m)\to(n',m')} =\hat{\chi}_{n',m'}(L u_{n,m})=\hat{\chi}_{n',m'}\circ L\circ\hat{\chi}'_{n,m}(u_{n,m})
\intertext{and}
&v_{n',m'} =\hat{\chi}_{n',m'}(Lu)=
\sum_{(n,m)} v_{(n,m)\to(n',m')} 
\end{align*}
where ${\chi}'_{n,m}$ is that defined in the proof of Lemma \ref{lem:xnm_trace} and $\hat{\chi}'_{n,m}$ denotes the convolution operator by the function $(2\pi)^{-1}\Fourier^{-1}{\chi}'_{n,m}$. Since $(1-\smoothing_{k})$ on $\cB^{r,p}(\real^2)$ is bounded uniformly in $k$  and cut off the low-frequency components, it suffices to show 
\begin{equation}\label{eq:required}
\sum_{n',m'}(2^{rm'}\|v_{n',m'}\|_{L^{2p}})^{2p}\le C_0
\sum_{n,m} (2^{rm}\|u_{n,m}\|_{L^{2p}})^{2p}
\end{equation}
assuming that  $u_{n,m}$ vanishes when $2^m \langle n\rangle^2\le k$  for some large $k$. 

We estimate the operator norm of $\hat{\chi}_{n',m'}\circ L\circ\hat{\chi}'_{n,m}$ on $L^{2p}(\real^2)$. Let us set
\begin{equation}\label{eq:Delta1}
\Delta_{1}(n,n')=\begin{cases}
1, & \mbox{if }|n-n'|\le 3;\\
\max\{n,n'\}, & \mbox{otherwise}
\end{cases}
\end{equation}
and 
\begin{align*}
&\Delta_{2}(n,m,n',m')\\
&\qquad=\begin{cases}
1, &\mbox{if either $m'=0$ or $2^{m'}\langle n'\rangle^{2}\le 2^{m+4}\langle n\rangle^{2}$;}\\
\max\{2^{m}\langle n\rangle^{2},2^{m'}\langle n'\rangle^{2}\}, & \mbox{otherwise}.
\end{cases}
\end{align*}
We are going to prove two estimates: One is that, for any $\nu>0$, there exists a constant $C_\nu>0$, depending only on $\nu$ and the constants $K_m$'s in (\ref{eq:py}), such that 
\begin{equation}\label{eq:basic_est}
\|\hat{\chi}_{n',m'}\circ L\circ\hat{\chi}'_{n,m}\|_{L^{2p}}\le C_\nu \Delta_{1}(n,n')^{-\nu}
\end{equation}
for any combination  $(n,m,n',m')$.
The other is that, for any $\nu>0$, there exists a constant $C(A,\varphi,\nu)$, depending $\nu$, $A$ and $\varphi$, such that 
\begin{equation}\label{eq:vnm}
\|\hat{\chi}_{n',m'}\circ L\circ\hat{\chi}'_{n,m}\|_{L^{2p}}\le C(A,\varphi,\nu)\cdot\Delta_{1}(n,n')^{-\nu}\cdot\Delta_{2}(n,m,n',m')^{-\nu}
\end{equation}
for any combination  $(n,m,n',m')$.

The required estimate (\ref{eq:required}) will follow from (\ref{eq:basic_est}) and (\ref{eq:vnm}). 
By using H\"older inequality, we see that the left hand side of (\ref{eq:required}) is bounded by
\[
\sum_{n',m'}2^{2prm'}\!\left\|\sum_{n,m}v_{(n,m)\to (n',m')}\right\|_{L^{2p}}^{2p}\le 
\sum_{n',m'}\sum_{n,m}\Delta_{n,m,n',m'}\cdot 2^{2prm}\left\|v_{(n,m)\to (n',m')}\right\|_{L^{2p}}^{2p}
\]
where
\[
\Delta_{n,m,n',m'}=C_0 \cdot 2^{2pr(m'-m)}\cdot 2^{|m'-m|}\cdot \langle n-n'\rangle^{2}
\]
with $C_0>0$ a constant depending only on $p$ and $r$. 
The estimate (\ref{eq:vnm}) with large $\nu$ implies that the components $v_{(n,m)\to(n',m')}$ is very small 
if $m'>0$ and $2^{m'}\langle n'\rangle^{2}> 2^{m+4}\langle n\rangle^{2}$. (Recall that we suppose 
$u_{n,m}$ vanishes when $2^m\langle n\rangle^2\le k$  for some large $k$.)
Hence the sum on the right hand side above over such combinations $(n,m,m',m')$ are negligible or more precisely bounded by $c\cdot \sum_{n,m} (2^{rm}\|u_{n,m}\|_{L^{2p}})^{2p}$ and we may let the constant $c>0$ be arbitrarily small by letting $k$ large. 
To the remaining components for which either $m'=0\le m$ or $2^{m'}\langle n'\rangle^{2}\le 2^{m+4}\langle n\rangle^{2}$ holds, 
we apply the estimate (\ref{eq:basic_est}) with large $\nu$. Then we obtain the required estimate (\ref{eq:required}) by elementary computation. 

To prove (\ref{eq:basic_est}) and (\ref{eq:vnm}), we look into the integral kernel of  $\hat{\chi}_{n',m'}\circ L\circ\hat{\chi}'_{n,m}$ and estimate it by using integration by parts.  Though the following argument is elementary and already presented in \cite{BT08}, we give it to some detail for completeness. (We will use a similar argument later, where we will omit the proof.) 
To begin with, let us make the following observation which motivates the definitions of $\Delta_1(\cdot)$ and $\Delta_2(\cdot)$: There exists a small constant $c>0$ such that, for any $(\xi',\eta')\in \supp \chi_{n',m'}$ and any $(\tilde{\xi},\tilde{\eta})\in DA^{\dag}_w(\supp \chi'_{n,m})$ with $w\in V$, we have
\begin{equation}\label{eq:ydiff}
|\eta'-\tilde{\eta}|\ge c \max\{|n|,|n'|\}\quad \mbox{ if \quad$|n-n'|\ge 4$}
\end{equation}
and
\begin{align}\label{eq:xdiff}
|\xi'-\tilde{\xi}|\ge c \max\{2^{m}\langle n\rangle^{2}&,2^{m'}\langle n'\rangle^{2}\}\quad \mbox{if \; $m'>0$\; and\; $2^{m'}\langle n'\rangle^{2}> 2^{m+4}\langle n\rangle^{2}$}.
\end{align}

Next let us write the operator $\hat{\chi}_{n',m'}\circ L\circ\hat{\chi}'_{n,m}$  as an integral operator 
\[
\hat{\chi}_{n',m'}\circ L\circ\hat{\chi}'_{n,m} u(z')=
(2\pi)^{-2}\int K(z',z)u(z) dz
\]
with the integral kernel
\begin{align}\label{eq:K}
K(&z',z)=\\
& \int e^{i\theta'\cdot (z'-w)+i\theta \cdot (A^{-1}(w)-z)}\cdot \chi_{n',m'}(\theta') \cdot \chi'_{n,m}(\theta) \cdot \varphi(A^{-1}(w)) d\theta d\theta' dw.\notag
\end{align}
To apply integration by parts, we consider the differential operators
\[
\mathcal{D}_1=\frac{1-i(\eta -\eta') \cdot \partial_y}{1+|\eta -\eta'|^2},
\quad
\mathcal{D}_2=\frac{1-i(DA^\dag_w\theta -\theta') \cdot \partial_w}{1+|DA^\dag_w\theta -\theta' |^2}\]
expressed in the coordinates $\theta=(\xi,\eta)$, $\theta'=(\xi',\eta')$ and $w=(x,y)$. These satisfy
\[
\mathcal{D}_j e^{i(\theta \cdot A^{-1}(w)-\theta'\cdot w)}=e^{i(\theta \cdot A^{-1}(w)-\theta'\cdot w)},\quad j=1,2.
\]
(For the case $j=1$, note that $A$ is written in the form (\ref{eq:A}).)
Hence 
\begin{align*}
\int e^{i(\theta' \cdot A^{-1}(w)-\theta\cdot w)} \Phi(w) dw&=
\int \left(\mathcal{D}_je^{i(\theta' \cdot A^{-1}(w)-\theta\cdot w)}\right)\Phi(w) dw \\
&= \int e^{i(\theta' \cdot A^{-1}(w)-\theta\cdot w)}\cdot  \left({}^t\mathcal{D}_j\right)\Phi(w) dw
\end{align*}
for $j=1,2$, 
where ${}^t\mathcal{D}_j$ denotes the transpose of $\mathcal{D}_j$ with respect to the $L^2$ inner product. 
We apply this formula with $j=1$ for several times if $|n-n'|\ge 4$ and then apply that with $j=2$ for several times if $m'>0$ and $2^{m'}\langle n'\rangle^{2}> 2^{m+4}\langle n\rangle^{2}$. As the result, we  get the expression of the form 
\[
K(z',z)= 
\int e^{i\theta'\cdot (z'-w)+i\theta \cdot (A^{-1}(w)-z)}\cdot 
\Psi(w,\theta, \theta') dw d\theta d\theta'
\] 
where the integration with respect to the variables $\theta'$ and $\theta$ are taken over the supports of  $\chi_{n',m'}$ and $\chi'_{n,m}$ respectively.
Using the estimates (\ref{eq:ydiff}) and (\ref{eq:xdiff}), we see, for arbitrarily large $\nu\ge 1$ and for any integers $\alpha,\alpha',\beta,\beta'\ge 0$, that
\[
|\partial^{\alpha}_{\xi}\partial^{\beta}_\eta \partial^{\alpha'}_{\xi'}\partial^{\beta'}_{\eta'}\Psi(w,\theta, \theta')|
\le 
\frac{C_{\nu,\alpha,\beta,\alpha',\beta'}\cdot \Delta_{1}(n,n')^{-\nu}\cdot\Delta_{2}(n,m,n',m')^{-\nu}} 
{\langle n\rangle^{\beta}\cdot \langle n'\rangle^{\beta'}
 \cdot \langle 2^m\langle n\rangle^2 \rangle^{\alpha}\cdot 
\langle 2^{m'}\langle (n')^2\rangle \rangle^{\alpha'} 
}
\]
where the constants $C_{\nu,\alpha,\beta,\alpha',\beta'}$ depend on $A$ and $\varphi$ but not on $n$, $m$, $n'$ nor~$m'$.
This implies that, for arbitrarily large $\nu>0$, we have
\begin{multline}\label{eq:estK}
|K(z',z)|\le C(A,\varphi,\nu)\cdot\Delta_{1}(n,n')^{-\nu}\cdot\Delta_{2}(n,m,n',m')^{-\nu}\\ \cdot \int \rho_{n',m'}^{(\nu)}(z'-w) \cdot \rho_{n,m}^{(\nu)}(A^{-1}(w)-z) dw
\end{multline}
where 
\[
\rho_{n,m}^{(\nu)}(x,y)=
2^m \langle n\rangle^{3}\cdot \langle 2^m \langle n\rangle^{2} |x-x'|\rangle^{-\nu}\cdot
\langle \langle n\rangle |y-y'|\rangle^{-\nu}.
\]
Hence we conclude the estimate (\ref{eq:vnm}) by Young's inequality.  
Note that if we did not apply integration by parts using $\mathcal{D}_2$, we obtain the estimate
\begin{equation}\label{eq:estK2}
|\partial^\alpha_{\xi}\partial^\beta_\eta \partial^{\alpha'}_{\xi'}\partial^{\beta'}_{\eta'}\Psi(w,\theta, \theta')|
\le 
\frac{C'_{\nu,\alpha,\beta,\alpha',\beta'}\cdot \Delta_{1}(n,n')^{-\nu}} 
{\langle n\rangle^{\beta}\cdot \langle n'\rangle^{\beta'}
 \cdot \langle 2^m\langle n\rangle^2 \rangle^{\alpha}\cdot 
\langle 2^{m'}\langle (n')^2\rangle \rangle^{\alpha'} 
}
\end{equation}
where the constants $C'_{\nu,\alpha,\beta,\alpha',\beta'}$ depend on $\nu$ and the constants $K_m$'s in (\ref{eq:py}) but not on $A$, $\varphi$, $n$, $m$, $n'$ nor~$m'$. Hence we obtain (\ref{eq:basic_est}) by a parallel argument. 
\end{proof}
\begin{lem}
\label{lm:pu} Let $U\subset\real^{2}$ be a bounded region. 
Let $\rho_{i}:\real^{2}\to[0,1]$, $1\le i\le I$, be a finite set of  $C^{\infty}$ functions with compact supports such that $\sum_{i=1}^{I}\rho_i (x)\equiv 1$ for $x\in U$. 
Suppose that the functions $\rho_i$ satisfy the condition (\ref{eq:py}) with $\varphi=\rho_i$ for some constants $K_m$ for $m\ge 0$. 
Then there exists a constant $C_0>0$, which depends only on $p$, $r$ and the constants $K_m$, such that, for sufficiently large $k>0$ (depending on the functions $\rho_j$), we have 
\[
\sum_{i=1}^I\|\rho_{i}\cdot(1-\smoothing_{k})u\|_{r,p}^{2p}\le C_{0}\|u\|_{r,p}^{2p}
\]
and 
\[
\|(1-\smoothing_{k})u\|_{r,p}^{2p}\le C_{0} \mu^{2p-1}\cdot\sum_{i=1}^I\|\rho_{i}\cdot u\|_{r,p}^{2p}
\]
for any $u\in \cB^{r,p}(U)$, where $\mu$ is the intersection multiplicity
of the subsets 
\[
X_i:=\{ x\in \real\mid (x,y)\in \supp \rho_{i} \mbox{ for some $y\in \real$}\} \quad\mbox{for $1\le i\le I$}.
\] 
\end{lem}
\begin{proof}
To get the claims of the lemma, we reconsider the argument in the proof of 
Proposition \ref{lm:local1} in the case $A=\mathrm{Id}$ and $\varphi=\rho_i$, and pay extra attention to the localized property of the kernel of  $\hat{\chi}_{n',m'}\circ L\circ\hat{\chi}'_{n,m}$ given in (\ref{eq:estK}) and (\ref{eq:estK2}). We omit the detail of the proof as it is easy to provide. 
\end{proof}  

\subsection{An $L^{p}$ estimate using transversality }
The next lemma is the key step of the argument in the proof of Theorem  \ref{thm:spectrum}. 
\begin{prop}\label{lm:main}
Let $S(i)$ and $E(i)$, 
$1\le i\le M$, be real numbers such that $|S(i)|\le \gamma_0\theta_0$ and $E(i)\ge \ell$. For a $p$-tuple 
$\bi=(\bi(1),\bi(2),\cdots,\bi(p))\in\{1,2,\cdots,M\}^{p}$,
we define 
\[
S(\bi):=\sum_{k=1}^{p}S(\bi(k)),\quad E(\bi):=\left(\sum_{k=1}^pE(\bi(k))^{-1}\right)^{-1}
\]
and set 
\begin{equation}\label{eq:Delta}
\Delta=\max_{\xi}\sum_{\bi\in\{1,2,\cdots,M\}^{p}}\langle(E(\bi)/\theta_{0})|\xi-S(\bi)|\rangle^{-r}.
\end{equation}
Then there exists a constant $C_0>0$, independent of $S(i)$ and $E(i)$, such that, for sufficiently large $k>0$, we have
\begin{equation}\label{eq:main_claim}
\left\Vert \sum_{i=1}^{M}
(1-\smoothing_{k})u_{i}\right\Vert_{r,p}^{2p}
\le
 C_{0}
\max\left\{\frac{M^{2p-1}}
{{\displaystyle \min_{1\le i\le M}} E(i)^{2pr}},\; M^{p-1}\Delta\right\}\sum_{i=1}^{M}\|u_{i}\|_{r,p,S(i),E(i)}^{2p}
\end{equation}
for any $u_{i}\in \cB_{S(i),E(i)}^{r,p}(\real^2)$.  
\end{prop}
\begin{proof} Inspecting the supports of the functions ${\chi}_{n,m}$ and ${\chi}_{n',m',S(i),E(i)}$, we find a constant $c_0>0$, independent of $S(i)$ and $E(i)$, such that
\[
 {\chi}_{n,m} \cdot  {\chi}_{n',m',S(i),E(i)}\equiv 0\qquad (\mbox{or } \;\;
 \hat{\chi}_{n,m} *  \hat{\chi}_{n',m',S(i),E(i)}= 0)
\]
if  $|n-n'|\ge 3$ or if $m> 0$ and $m\ge m'-\log E(i)/\log 2+c_0$.
From Lemma~\ref{lem:Fxnm2}, the $L^1$ norm of the functions 
$\hat{\chi}_{n,m}$ and $\hat{\chi}_{S(i), E(i),n',m'}$ are bounded by a constant independent of, $n$, $m$, $n'$, $m'$, $S(i)$ and $E(i)$ and therefore so are the operator norms of the convolution operators with these functions on $L^{2p}(\real^2)$. 
Hence, by using H\"older inequality twice, we obtain that 
\begin{align}
\label{eq:estimate_m_nonzero}
\sum_{n}\sum_{m>0}&
\left(2^{rm}
\left\|\hat{\chi}_{n,m}\left(\sum_{i=1}^{M}u_{i}\right)\right\|_{L^{2p}}\right)^{2p}  
\\
&\le 
M^{2p-1}
\sum_{n}\sum_{m>0}\sum_{i=1}^{M}2^{2prm}\|\hat{\chi}_{n,m}(u_{i})\|_{L^{2p}}^{2p}
\notag \\
&\le 
M^{2p-1}
\sum_{i=1}^{M}\sum_{n}\sum_{m> 0}
\left\|\sum_{n'}\sum_{m'}2^{rm}\hat{\chi}_{n,m}\circ \hat{\chi}_{n',m',S(i),E(i)}(u_{i})\right\|_{L^{2p}}^{2p}
\notag \\
& \le 
 \frac{C_0\cdot M^{2p-1}}{\min_{1\le i\le M} E(i)^{2pr}}\sum_{i=1}^{M}
\sum_{n'}\sum_{m'}
2^{2prm'}\|\hat{\chi}_{n',m',S(i),E(i)}(u_{i})\|_{L^{2p}}^{2p}.
\notag
\end{align}
Notice that we excluded the components with $m=0$ in the estimate above. 
Below we give an estimate on the components with $m=0$, which is more essential.  Note that we may (and will) suppose  that $|n|$ is large, by letting $k$ (in the definition of $\mathcal{K}_k$) be larger if necessary. 

For a $p$-tuple $\bi=(\bi(1),\bi(2),\cdots,\bi(p))\in\{1,2,\cdots,M\}^{p}$, we write 
\[
u_{\bi}=\prod_{k=1}^p\hat{\chi}_{n,0} (u_{\bi(k)})
\]
and estimate the $L^2$ norm of  $\hat{\chi}_{S(\bi),E(\bi),\tilde{n},\tilde{m}}*u_\bi$ for integers $\tilde{n}$ and $\tilde{m}\ge 0$. 
Since the support of $\Fourier u_{\bi}$ is contained in the subset 
\[
 \sum_{k=1}^p \supp \chi_{n,0}:=\left\{\left.\sum_{k=1}^p x_i \;\right|\; x_i\in \supp \chi_{n,0}\right\}\subset \real^2,
\]
we have $\hat{\chi}_{S(\bi),E(\bi),\tilde{n},\tilde{m}}*u_\bi=0$  unless 
\begin{equation}\label{eq:relnn}
||\tilde{n}|^2-p|n|^2|\le  p(2|n|+1)+2|\tilde{n}|+1.
\end{equation}
We henceforth suppose that $\tilde{n}$ satisfies  (\ref{eq:relnn}).
Since we assume $|n|$ is large, this implies that the ratio  $\tilde{n}/n$ is close to $\sqrt{p}$ and we have
$|\tilde{n}-\sqrt{p}n|\le 3(\sqrt{p}+1)$.  

For convenience in the argument below, we introduce the functions
\[
\zeta_{n,m,S,E}(\xi,\eta):={\chi}_{n,m,S,E}(\xi,\eta)\cdot \chi(\theta_0^{-1}\cdot \langle n\rangle^{-2}\cdot \xi)
\]
on $\real^2$, which satisfy 
\[
\sum_{m\ge 0} \zeta_{n,m,S,E}(\xi,\eta)=\rho_n(\eta)\cdot \chi(\theta_0^{-1}\cdot \langle n\rangle^{-2}\cdot \xi)=\chi_{n,0}(\xi,\eta).
\]
We write $\hat{\zeta}_{n,m,S,E}$ for the convolution operator by the function $(2\pi)^{-1}\Fourier^{-1}\zeta_{n,m,S,E}$. Then we have
\begin{equation}\label{eq:zeta}
\sum_{m\ge 0} \hat{\zeta}_{n,m,S,E}(u)=\hat{\chi}_{n,0}(u)
\end{equation}
and also
\[
\|\hat{\zeta}_{n,m,S,E}(u)\|_{L^{2p}}\le C_0 \|\hat{\chi}_{n,m,S,E}(u)\|_{L^{2p}}
\]
for a constant $C_0>0$ which depend only on the choice of the function $\chi(\cdot)$. 

For a sequence $\bm=(\bm(1), \cdots, \bm(p))\in (\integer_{\ge 0})^p$ of non-negative integers, put 
\[
|\bm|=\max_{1\le i\le p} \bm(k).
\]
By inspecting the position of the supports of $\chi_{n,m,S,E}(\cdot)$ in the $\xi$-coordinate, we find a constant $C>0$, which depends only on $p$,  such that, if $|\bm|<\tilde{m}-C$, we have
\[
\supp \hat{\chi}_{\tilde{n},\tilde{m},S(\bi),E(\bi)} \cap \left(\sum_{k=1}^p \supp \hat{\chi}_{n,\bm(k),S(\bi(k)),E(\bi(k))}\right)=\emptyset   
\]
and hence 
\[
\hat{\chi}_{\tilde{n},\tilde{m},S(\bi),E(\bi)} *\left(\prod_{k=1}^{p}\hat{\zeta}_{n,\bm(k),S(\bi(k)),E(\bi(k))}(u_{\bi(k)})\right)=0.
\]
From this and (\ref{eq:zeta}), we have   
\[
\|\hat{\chi}_{\tilde{n},\tilde{m},S(\bi), E(\bi)}(u_{\bi})\|_{L^{2}}^{2}  = \left\|
\sum_{|\bm|\ge \tilde{m}-C}\prod_{k=1}^{p}\hat{\zeta}_{n,\bm(k),S(\bi(k)),E(\bi(k))}(u_{\bi(k)})
\right\|_{L^{2}}^{2}.
\]
By using Schwarz and H\"older inequality, we continue
\begin{align*}
  & \le C_{0}\sum_{|\bm|\ge \tilde{m}-C}
 2^{|\bm|}\left\|\prod_{k=1}^{p}\hat{\zeta}_{n,\bm(k),S(\bi(k)), E(\bi(k))}(u_{\bi(k)})\right\|_{L^{2}}^{2}\\
  & \le C_{0}\sum_{|\bm|\ge \tilde{m}-C}
 2^{|\bm|}\prod_{k=1}^{p}\left\|\hat{\zeta}_{n,\bm(k),S(\bi(k)),E(\bi(k))}(u_{\bi(k)})\right\|_{L^{2p}}^{2}
\end{align*}
and further 
\begin{align*}
 & \le C_{0}2^{-r\tilde{m}}\sum_{\bm}\prod_{k=1}^{p}\left(2^{(r+1)\bm(k)}\|\hat{\zeta}_{n,\bm(k),S(\bi(k)),E(\bi(k))}(u_{\bi(k)})\|_{L^{2p}}^{2}\right)\\
 & \le C_{0}2^{-r\tilde{m}}\prod_{k=1}^{p}\left(\sum_{m=0}^\infty 2^{(r+1)m}\|\hat{\chi}_{n,m,S(\bi(k)),E(\bi(k))}(u_{\bi(k)})\|_{L^{2p}}^{2}\right)
\end{align*}
where (and henceforth in the proof below) $C_0>0$ denotes constants depending only on $p$ and $r$ and its values may be different from place to place. 
We therefore conclude
\begin{align}\label{eq:core_estimate}
&2^{r\tilde{m}}\|\hat{\chi}_{\tilde{n},\tilde{m},S(\bi),E(\bi)}( u_{\bi})\|_{L^{2}}^{2}\\ 
&\qquad\qquad \le C_{0}\prod_{k=1}^{p}\left(\sum_{m=0}^\infty 2^{2rm}\|\hat{\chi}_{n,m,S(\bi(k)),E(\bi(k))}(u_{\bi(k)})\|_{L^{2p}}^{2}\right)\notag
\end{align}
where we used the fact $2r\ge r+1$ that follows from (\ref{eq:choice_r}). 
 
Now we are going to prove the conclusion of the proposition. Recall the quantity $\Delta$ defined in (\ref{eq:Delta}) and write 
\[
W_{\bi}(\xi,\eta)=\big\langle(E(\bi)/\theta_{0})|\xi/\langle\eta\rangle-S(\bi)|\big\rangle^{r/2}.
\]
Then we have
\begin{align*}
\left\|\hat{\chi}_{n,0}\left(\sum_{i=1}^M u_{i}\right)\right\|_{L^{2p}}^{2p} 
&=
\left\|
\sum_{\bi} u_{\bi} \right\|_{L^{2}}^{2}
\le\sum_{\bi,\bj}\|W_{\bj}^{-1}\cdot W_{\bi}\cdot\Fourier u_{\bi}\|_{L^{2}}\cdot \|W_{\bi}^{-1}\cdot W_{\bj}\cdot\Fourier u_{\bj}\|_{L^{2}}\\
& \le\sum_{\bi,\bj}\|W_{\bj}^{-1}\cdot W_{\bi}\cdot\Fourier u_{\bi}\|_{L^{2}}^{2}\\
 & \le\left\|\sum_{\bj}W_{\bj}^{-2}\right\|_{\infty}\cdot \sum_{\bi}\|W_{\bi}\cdot\Fourier u_{\bi}\|_{L^{2}}^{2}\le\Delta\cdot \sum_{\bi}\|W_{\bi}\cdot\Fourier u_{\bi}\|_{L^{2}}^{2}.
\end{align*}
Since $W_{\bi}(\xi,\eta)\le C_0 2^{r\tilde{m}/2}$ on the support of $ \chi_{\tilde{n},\tilde{m},S(\bi(k)),E(\bi(k))}$, we have from  (\ref{eq:core_estimate}) that
\begin{multline*}
\left\| W_\bi \cdot \Fourier u_\bi\right\|_{L^{2}}^{2}\le
C_0 \sum_{\tilde{n}:|\tilde{n}-n|\le 3(\sqrt{p}+1)} \;
\sum_{\tilde{m}=0}^{\infty}2^{r\tilde{m}}\|\hat{\chi}_{\tilde{n},\tilde{m},S(\bi),E(\bi)}(u_{\bi})\|_{L^{2}}^{2}\\
\le C_{0}\sum_{\tilde{n}:|\tilde{n}-n|\le 3(\sqrt{p}+1)}\;\prod_{k=1}^{p}\left(\sum_{m=0}^\infty 2^{2rm}\|\hat{\chi}_{n,m,S(\bi(k)),E(\bi(k))}(u_{\bi(k)})\|_{L^{2p}}^{2}\right).
\end{multline*}
From the last two inequalities, we deduce
\begin{align}\label{eq:chi0}
&\sum_{n}\left\|\hat{\chi}_{n,0}\left(\sum_{i=1}^Mu_{i}\right)\right\|_{L^{2p}}^{2p}\\
&\le 
C_{0}\Delta\cdot  \sum_{n}\sum_{\tilde{n}:|\tilde{n}-n|\le 3(\sqrt{p}+1)}\sum_{\bi}\prod_{k=1}^{p}\left(\sum_{m=0}^\infty 2^{2rm}\|\hat{\chi}_{n,m,S(\bi(k)),E(\bi(k))}(u_{\bi(k)})\|_{L^{2p}}^{2}\right)
\notag\\
&\le 
C_{0}\Delta\cdot  \sum_{n}
\left(\sum_{i=1}^M\sum_{m=0}^\infty 2^{2rm}\|\hat{\chi}_{n,m,S(i),E(i)}(u_{i})\|_{L^{2p}}^{2}\right)^p
\notag\\
& \le
C_{0}\Delta M^{p-1}\cdot \sum_{i=1}^M \sum_{n}\sum_{m=0}^\infty 2^{2prm}\|\hat{\chi}_{n,m,S(i),E(i)}(u_{i})\|_{L^{2p}}^{2p}\quad \mbox{by H\"older inequality.}\notag
\end{align}
Finally note that 
\[
\left\|\sum_{i=1}^{M}(1-\smoothing_{k})u_{i}\right\|_{r,p}^{2p}
\le \sum_{n,m}{}^*
\left(2^{rm}
\left\|\hat{\chi}_{n,m}\left(\sum_{i=1}^{M}u_{i}\right)\right\|_{L^{2p}}\right)^{2p}  
\]
where the sum $\sum^*_{n,m}$ is taken over $n$ and $m\ge 0$ such that$2^m\langle n^2\rangle\ge k$. By (\ref{eq:estimate_m_nonzero}) and (\ref{eq:chi0}), we obtain the conclusion of the proposition. 
\end{proof}
\begin{cor}\label{cor:bdd}
In Proposition \ref{lm:local1}, the operator $L$ is a bounded operator from $\cB^{r,p}(V)$ to $\cB^{r,p}(A(\supp \varphi))$. 
\end{cor}
\begin{proof}It suffices to show that the inclusion $\iota:\cB^{r,p}_{S,E}(A(\supp \varphi))\to \cB^{r,p}(A(\supp \varphi))$ is bounded. This follows from Proposition \ref{lm:main} applied to the trivial case $M=1$. 
\end{proof}
\section{Proof of Theorem \ref{thm:spectrum}}\label{sec:pf1}
Below we set up a system of local charts on $X_f$ so that the flow $T^t_f$ looks smooth in each of them and then introduce the Banach space $\cB^{r,p}(X_f)$ using such local charts and the Banach space $\cB^{r,p}(\real^2)$.  Once we have done with these, the proof of Theorem \ref{thm:spectrum} is not very difficult and  obtained basically by applying the propositions (especially Proposition \ref{lm:main}) to the transfer operators induced on the local charts. Unfortunately a slight combinatorial  complication is caused by the fact that we admit the "exceptional set" $\mathcal{E}$ in the definition of $\mathcal{G}(J,n,\varepsilon,\delta;p)$. In order to present the idea of the proof clearly, we first prove the conclusion of the theorem assuming a stronger condition where $\mathcal{E}=\emptyset$ in Subsection \ref{ss:pre} and then explain how we modify the argument to obtain the theorem in Subsection \ref{ss:pf_Thspec}.

\subsection{System of local charts on $X_{f}$ and the definition of $\cB^{r,p}(X_{f})$}
\label{ss:defcB}
To begin with, we take two small real numbers $\eta_{0}>0$ and $\delta_{0}>0$
and consider the open rectangle 
\[
R=(-\eta_{0},\eta_{0})\times(4\delta_{0},7\delta_{0})\subset Q=(-3\eta_{0},3\eta_{0})\times(0,11\delta_{0}).
\]
For each $a=(x_{0},y_{0})\in X_{f}$, we consider the two mappings
\begin{align*}
&\tilde{\kappa}_{a}:Q\to S^{1}\times\real,\quad\tilde{\kappa}_{a}(x,y)=(x_{0}+x,y_{0}+y).
\intertext{and}
&\kappa_{a}:=\pi\circ \tilde{\kappa}_a:Q\to X_{f}
\end{align*}
where
\begin{equation}\label{eq:proj}
\pi:S^{1}\times\real_{+}\to X_{f},\quad\pi(x,y)=(\tau^{n(x,y;f)}(x),\,y-f^{(n(x,y;f))})
\end{equation}
and $\real_+=\{s\in \real\mid s\ge 0\}$. 
(See Figure \ref{fig:kappa}.) 
We suppose that $\eta_{0}$ and $\delta_{0}$ are so small that both of $\kappa_{a}$ and $\tilde{\kappa}_{a}$ are injective for any $a\in X_{f}$.

\begin{figure}[htbp]
\begin{center}
\begin{overpic}[scale=0.4]{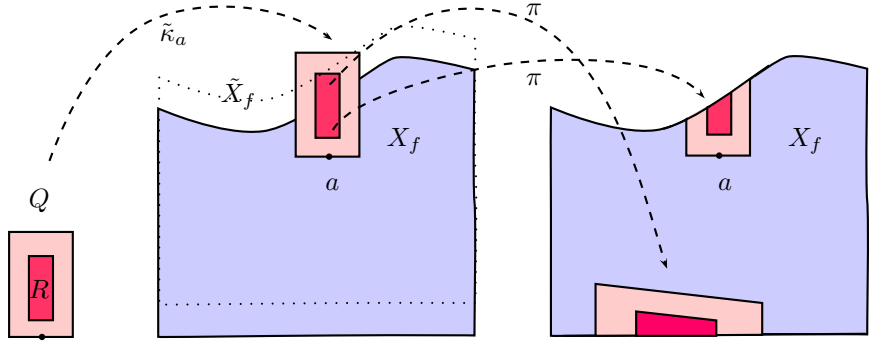}
\put(3,16){$Q$}
\put(3,5.5){$R$}
\put(44,23){$X_f$}
\put(90,23){$X_f$}
\put(60,38){$\pi$}
\put(60,30){$\pi$}
\put(25,28){$\tilde{X}_f$}
\put(18,35){$\tilde{\kappa}_a$}
\put(37,18){$a$}
\put(82,18){$a$}
\end{overpic}
\end{center}
\caption{The mappings $\tilde{\kappa}_a$, $\pi$ and $\kappa_a$.}
\label{fig:kappa}
\end{figure}

Next we take a finite subset $A$ of $X_{f}$ so that the images $\tilde{\kappa}_{a}(R)$
for $a\in A$ cover the subset 
\[
\tilde{X}_{f}:=\{(x,y)\in S^{1}\times\real_+\mid 5\delta_{0}\le y\le f(x)+6\delta_{0}\}.
\]
Letting $\delta_0$ and the ratio $\eta_{0}/\delta_{0}$ be small, we may  and do assume that
the intersection multiplicity of $\{\tilde{\kappa}_{a}(R)\}_{a\in A}$
is bounded by an absolute constant (say, by $4$). 

We define the Banach space $\cB^{r,p}(X_{f})$ as follows. 
We suppose that the product space $\bigoplus_{a\in A}\cB^{r,p}(R)$ is a Banach space  with the norm  
\[
\|\mathbf{u}\|_{r,p}=\left(\sum_{a\in A}\|u_{a}\|_{r,p}^{2p}\right)^{1/2p}\quad \mbox{for $\mathbf{u}=(u_{a})_{a\in A}\in\bigoplus_{a\in A}\cB^{r,p}(R)$.}
\]
Then the operator   
\begin{equation}\label{eq:Pi_extended}
\Pi:\bigoplus_{a\in A}\cB^{r,p}(R)\to L^2(X_{f}),\quad\Pi((\varphi_{a})_{a\in A})=\sum_{a\in A}\varphi_{a}\circ\kappa_{a}^{-1}
\end{equation}
is bounded because $\cB^{r,p}(R)\subset \cB^{r,2}(R)\subset L^2(R)$.
\begin{defn}
Let $\cB^{r,p}(X_{f})\subset L^2(X_f)$ be the image of  (\ref{eq:Pi_extended}). This is a Banach space with respect to the norm
 \[
\|u\|_{\cB^{r,p}}=\inf\left\{\|\mathbf{u}\|_{r,p}\;\left|\; u=\Pi(\mathbf{u}),\mathbf{u}\in\bigoplus_{a\in A}\cB^{r,p}(R)\right.\right\}.
\]
The operator $\Pi$ in (\ref{eq:Pi_extended}) is  then restricted to a bounded operator 
\[
\Pi:\bigoplus_{a\in A}\cB^{r,p}(R)\to \cB^{r,p}(X_{f})
\] 
with operator norm $1$. 
\end{defn}
We next define a bounded operator $\mathbf{I}:\cB^{r,p}(X_{f})\to \bigoplus_{a\in A}\cB^{r,p}(R)$ which makes the following diagram with $t=6\delta_0$ commutes:
\begin{equation}\label{cd:I}
\begin{tikzcd}
 &\bigoplus_{a\in A}\cB^{r,p}(R) \arrow{d}{\Pi}\\
 \cB^{r,p}(X_{f}) \arrow{ru}{\mathbf{I}} \arrow{r}{\cL^{t}} &\cB^{r,p}(X_{f})
\end{tikzcd}
\end{equation}
\begin{rem}
It would be  preferable if we let $t=0$ and defined the operator $\mathbf{I}$ as the left inverse of $\Pi$. This may be possible but will not be easy. 
\end{rem}
Let  $\beta:S^{1}\times\real\to[0,1]$ be a smooth function defined 
by 
\[
\beta(x,y)=\begin{cases}
\chi(\delta_{0}^{-1}(y-f(x)-5\delta_{0})+1), & \mbox{ if $ f(x)+5\delta_0\le y$;}\\
1, & \mbox{ if $6\delta_{0}<y<f(x)+5\delta_0$;}\\
1-\chi(\delta_{0}^{-1}(y-5\delta_{0})+1), & \mbox{ if $y\le 6\delta_{0}$}
\end{cases}
\]
where $\chi$ is the function defined in (\ref{eq:chi}).
This function is taken so that it satisfies
\[
\beta(x,y)=\begin{cases}
0,&\mbox{on the outside of $\tilde{X}_{f}$;}\\
1,&\mbox{when $6\delta_0\le y\le f(x)+5\delta_0$}
\end{cases}
\]
and also
\[
\beta(x,f(x)+y)+\beta(x',y)=1\quad \mbox{for any $x,x'\in S^1$ and $0\le y\le 6\delta_0$.}
\]
We then take $C^{\infty}$ functions  $h_{a}:\real^{2}\to[0,1]$ 
 supported on $R$ for $a\in A$ so that\footnote{Here and henceforth, we suppose that ${h}_{a}\circ \tilde \kappa_a^{-1}$ is a $C^\infty$ function on $S^1\times \real$ which takes value $0$ on the outside of $\tilde{\kappa}_a(Q)$. } 
\[
\sum_{a}{h}_{a}\circ \tilde \kappa_a^{-1}\equiv\beta
\quad \mbox{on $S^{1}\times\real$.}
\]
For each $u\in C^\infty(X_f)$, we set
\[
\tilde{u}:\tilde{X}_f\to \complex, 
\quad \tilde{u}(x,y)=\begin{cases}
(\cL^{6\delta_0} u)(x,y),&\quad \mbox{ if $y\le 6\delta_0$;}\\
u(x,y-6\delta_0),&\quad \mbox{ if $y\ge 6\delta_0$.} 
\end{cases}
\]
Since $(\cL^{6\delta_0} u)(x,y)=u(x,y-6\delta_0)$
when  $6\delta_0\le y\le f(x)$, this is a smooth function on $\tilde{X}_f$. 
We define the operator $\mathbf{I}$ by   
\begin{equation}\label{eq:bI}
\mathbf{I}(u)=(u_a)_{a\in A}, \quad u_a= h_a\cdot (\tilde{u} \circ \tilde{\kappa}_{a}) \quad\mbox{for $u\in C^\infty(X_f)$}. 
\end{equation}
This operator extends to a bounded operator 
$\mathbf{I}:\cB^{r,p}(X_f)\to \bigoplus_{a\in A}\cB^{r,p}(R)$, as we will see in the next paragraph, and makes the diagram (\ref{cd:I}) commutes.

Next we introduce the operator  
\begin{equation}\label{eq:bLt}
\bL^{t}:=\mathbf{I}\circ \cL^{t-6\delta_0}\circ \Pi:\bigoplus_{a\in A}\cB^{r,p}(R)\to\bigoplus_{a\in A}\cB^{r,p}(R)
\end{equation}
for $t\ge 6\delta_0$. By applying Corollary \ref{cor:bdd} to each component, we see that this is a bounded operator.
(See Remark \ref{lem:boundedness} for more detail.)
Since $\bL^{6\delta_0}=\mathbf{I}\circ \Pi$ is bounded in particular, so is $\mathbf{I}:\cB^{r,p}(X_f)\to \bigoplus_{a\in A}\cB^{r,p}(R)$ from the definition of $\cB^{r,p}(X_f)$.
From (\ref{cd:I}), the diagrams  
\begin{equation}
\begin{CD}\bigoplus_{a\in A}\cB^{r,p}(R)@>{\bL^{t}}>>\bigoplus_{a\in A}\cB^{r,p}(R)\\
@V{\Pi}VV@V{\Pi}VV\\
\cB^{r,p}(X_{f})@>{\cL^{t}}>>\cB^{r,p}(X_{f})
\end{CD}\label{cd:lift}
\end{equation}
and 
\begin{equation}
\begin{CD}\bigoplus_{a\in A}\cB^{r,p}(R)@>{\bL^{t}}>>\bigoplus_{a\in A}\cB^{r,p}(R)\\
@A{\mathbf{I}}AA@A{\mathbf{I}}AA\\
\cB^{r,p}(X_{f})@>{\cL^{t}}>>\cB^{r,p}(X_{f})
\end{CD}\label{cd:lift2}
\end{equation}
commute for $t\ge 6\delta_0$. In particular, the operator $\cL^t:\cB^{r,p}(X_{f})\to \cB^{r,p}(X_{f})$ is  bounded provided $t\ge 6\delta_0$.
It is not difficult to check that the operators
\[
\cL^t:\cB^{r,p}(X_{f})\to \cB^{r,p}(X_{f})\quad\mbox{and}\quad
\bL^{t}:\bigoplus_{a\in A}\cB^{r,p}(R)\to \bigoplus_{a\in A}\cB^{r,p}(R),
\] 
have the same essential spectral radus and their peripheral eigenvalues on the outside of it coincide up to multiplicity. 
  
The operator $\bL^{t}$ for $t\ge 6\delta_0$ is expressed as a matrix of operators 
\begin{equation}
\bL^{t}(u_{a})_{a\in A}=\left(\sum_{a\in A}\cL_{a\to b}^{t}u_{a}\right)_{b\in A}.\label{eq:bLt_exp}
\end{equation}
Each component $
\cL_{a\to b}^{t}:\cB^{r,p}(R)\to \cB^{r,p}(R)$
is written in the form (\ref{eq:L_local}), {\it i.e.} $\cL_{a \to b}^t u=(\varphi \cdot u)\circ A$ with  
\begin{equation}\label{eq:Adef}
A=A^t_{a\to b}:R^t_{a\to b}\to \real^2 \quad \mbox{and}\quad \varphi=\varphi^{t}_{a\to b}(x,y):=h_b\circ A^t_{a\to b}(x,y)
\end{equation}
where 
\begin{align}
&A^t_{a\to b}(x,y)=
\kappa_b^{-1}\circ T^t_f\circ \kappa_a(x,y)\label{def:Aab}
\intertext{
and }
&R^t_{a\to b}=\{z\in R\mid 
T^t_f\circ \kappa_a(z)\in \kappa_b(R) \}.\label{def:Rab}
\end{align} 
\begin{rem}
The mapping $A^t_{a\to b}$ is defined only on a relatively small open subset $R^t_{a\to b}$ in $R$, which will be fragmentary in the direction transversal to the flow when $t$ is large. It is locally written in the form (\ref{eq:A}) with $E\ge 1$ and with $g$ a $C^\infty$ function satisfying $|g'(x)|\le \gamma_0\theta_0$. Though the function $\varphi^t_{a\to b}$ is defined only on $R^t_{a\to b}$, we may extend it to a $C^\infty$ function on $\real^2$ with support contained in $Q$, by letting $\tilde{h}:\real^2\to [0,1]$ be a $C^{\infty}$ function such that
\[
\tilde{h}(z)=\begin{cases} 1,&\quad\mbox{on $R$;}\\
0,&\quad\mbox{on the outside of $Q$}, 
\end{cases}
\]
and 
setting
\[
\varphi_{a\to b}^t(z)=\begin{cases}
\tilde{h}(z)\cdot  h_b\circ A^t_{a\to b}(z),&\quad 
\mbox{if $z\in Q$ and $T^t_f\circ \kappa_a(z)\in \kappa_b(R)$;}\\
0,&\quad \mbox{otherwise.}
\end{cases}
\]
In particular, 
  $\cL_{a\to b}^{t}$ is smooth on $R$ in the sense that  $\cL^t_{a\to b}(C^\infty_0(R))\subset C^\infty_0(R)$. 
\end{rem}

\subsection{Essential operator norm}
We introduce the notion of essential operator norm of a bounded operator. 
This notion is particularly convenient in our argument about the essential spectral radius. 
For a bounded operator $L:B\to B'$ between Banach spaces $B$ and $B'$, its essential operator norm, denoted by $\|L:B\to B'\|_{\ess}$, is the infimum of the operator norms of its perturbations by compact operators: 
\[
\|L:B\to B'\|_{\ess}:=\inf\{\|L-K:B\to B'\|\mid K:B\to B' \mbox{ is compact}\}.
\]
Obviously this is bounded by the operator norm $\|L:B\to B'\|$.
Since composition of a compact operator with a bounded operator is again compact, we have
\[
\|L'\circ L:B\to B''\|_{\ess}\le \|L':B'\to B''\|_{\ess}\cdot \|L:B\to B'\|_{\ess}.
\]
The essential spectral radius of $L:B\to B$ is bounded by its essential norm: 
\[
\rho_{\ess}(L|_{B})\le \|L^n|_{B}\|_{\ess}^{1/n}\le \|L|_{B}\|_{\ess}.
\]

Theorem \ref{thm:spectrum} will follow from the claim that, if $\varepsilon>0$ and if $f$ is sufficiently close to $f_0\in \mathcal{G}$, there exists some $t_*\ge 6\delta_0$ such that 
\begin{equation}\label{eq:essnorm}
\left\|\bL^{t_*}|_{\bigoplus_{a\in A}\cB^{r,p,q}(R)}\right\|_{\ess}\le \exp((\mu(f)+\varepsilon) t_*).
\end{equation}
Indeed, from Corollary \ref{cor:bdd}, we have,  for some $C>0$, that
\begin{equation}\label{eq:norm}
\left\|\bL^{t}|_{\bigoplus_{a\in A}\cB^{r,p,q}(R)}\right\|\le C\quad 
\mbox{for $6\delta_0\le t\le t_*+6\delta_0$}.
\end{equation}
(See Remark \ref{lem:boundedness} for more detail.) Since 
\begin{multline*}
\rho_{\ess}(\cL^t|_{\cB^{r,p}(X_f)})\le \|\cL^{nt}|_{\cB^{r,p}(X_f)}\|^{1/n}_{\ess}
= \|\Pi\circ \bL^{t-6\delta_0}\circ \mathbf{I}\|_{\ess}^{1/n}\\
\le   \|\Pi\|^{1/n}\cdot 
\|\bL^{t_*}\|_{\ess}^{\lfloor(nt-6\delta_0)/t_*\rfloor/n} 
\cdot \|\bL^{nt-\lfloor(nt-6\delta_0)/t_*\rfloor\cdot t_*}\|^{1/n}\cdot\|\mathbf{I}\|^{1/n},
\end{multline*}
we obtain the conclusion of Theorem \ref{thm:spectrum} by letting $n\to \infty$. In the following subsections, we  prove the claim (\ref{eq:essnorm}). 

\subsection{Reduction of the claim}\label{ss:reduce}
Below we show that the claim (\ref{eq:essnorm}) follows from the corresponding estimates on some localized transfer operators on local charts, to which we can apply Proposition \ref{lm:local1} and \ref{lm:main}.  We proceed in a few steps. 
First note that the claim (\ref{eq:essnorm}) follows if we show that
\begin{equation}\label{eq:reducedClaim1}
\|\cL_{a\to b}^t:{\cB}^{r,p}(R)\to {\cB}^{r,p}(R)\|_{\ess}\le C_0\exp((\mu(f)+\varepsilon)t)
\end{equation}
for sufficiently large $t>0$ and for all $a,b\in A$, with $C_0$ a constant independent of~$t$. 
To proceed, we take a finite family of $C^\infty$ functions 
\[
\{\rho_j^t:\real^2\to [0,1]\}_{j=1}^{J(t)}
\]
for each $t>0$, such that  $\sum_{j=1}^{J(t)} \rho_j^t\equiv 1$ on $R$ and that $\supp \rho_j^t\subset Q$. We assume that 
\begin{itemize}
\item $\rho_{j}^t$ satisfies  (\ref{eq:py}) with some constants $K_m>0$ uniform in $j$ and $t$,
\item the support $\rho_j^t$ is contained in a region of the form 
$I_{j}^t\times \real$ 
where $I_j^t$ is a closed interval on $\real$,  and 
\item the intersection multiplicity of $I_j^t$, $1\le j\le J(t)$, is bounded by $2$ (say). 
\end{itemize}
\begin{rem}
In the following subsections, we will assume that the length of the interval $I_j^t$ is very small when $t$ is large. It is important that the constants denoted by $C_0$ below do not depend on the choice of the functions $\rho_j^t$ (though they may depend on the constants $K_m$).
\end{rem}

Let us write the operator $\cL^t_{a\to b}$ as 
\begin{equation}\label{eq:rhoL}
\cL^t_{a\to b}=\sum_{j=1}^{J(t)} \mult(\rho_j^t)\circ \cL_{a\to b}^t
\end{equation}
where $\mult(\rho_j^t)$ denotes the multiplication operator by $\rho_j^t$.
By the second claim of Lemma \ref{lm:pu}, we see that the inequality (\ref{eq:reducedClaim1})  follows if we prove
\begin{equation}\label{eq:reduced_claim2}
\|\mult(\rho_j^t)\circ \cL_{a\to b}^t:{\cB}^{r,p}(R^t_{a\to b})\to {\cB}^{r,p}(R)\|_{\ess}\le C_0 \exp((\mu(f)+\varepsilon)t)
\end{equation}
for all $1\le j\le J(t)$, $a,b\in A$ and for sufficiently large $t>0$, with a constant $C_0$ independent of $t$ and $j$. 

We decompose the operator $\mult(\rho_j^t)\circ \cL_{a\to b}^t$ in (\ref{eq:reduced_claim2}) further. 
For $w\in (T_f^{t})^{-1}(b)$,  there is
a unique open neighborhood $U^{t}_{b,w}$ of the point $w+(0,6\delta_0)$ in $S^1\times \real_+$
 that is mapped bijectively onto $\kappa_b(R)$ by $T_{f}^{t}\circ \pi$. (Recall (\ref{eq:proj}) for the definition of $\pi$.)
We define
\[
R^t_{a\to b, w}:=\kappa_a^{-1}(\pi(U^{t}_{b,w}))\cap R \subset R^t_{a\to b}\subset \real^2.
\]
Then $R^t_{a\to b}$ is the disjoint union of $R^t_{a\to b,w}$ for $w\in (T_f^{t})^{-1}(b)$ though some of $R^t_{a\to b,w}$ will be empty. 
Correspondingly we define 
\begin{align*}
&A_{a\to b,w}^{t}=A_{a\to b}^{t}|_{R^t_{a\to b, w}}:R^t_{a\to b, w}\to R
\intertext{and, for $1\le j\le J(t)$,}
&\rho^t_{a\to b,w,j}:R^t_{a\to b,w}\to [0,1], \quad 
\rho^t_{a\to b,w,j}(z)=(h_b\cdot \rho_j^t)\circ A_{a\to b,w}.
\end{align*}  
Then the operator $\mult(\rho_j^t)\circ \cL_{a\to b}^t$ is written as the sum
\[
\mult(\rho_j^t)\circ \cL^{t}_{a,b}=
\sum_{w\in (T_f^{t})^{-1}(b)} \cL^{t}_{a\to b,w,j}:C^{\infty}_0(R)\to C^\infty_0(R)
\]
where $\cL^{t}_{a\to b,w,j}=0$ if $R^t_{a\to b,w}=\emptyset$ and otherwise
\[
\cL^{t}_{a\to b,w,j}u=(\rho^t_{a\to b,w,j}\cdot u)\circ (A_{a\to b,w}^{t})^{-1}.
\]
\begin{rem}Notice that the functions $\rho^t_{a\to b,w,j}$ satisfy the condition (\ref{eq:py}) with some constants $K_m>0$ uniform for $a$, $b$, $w$, $j$ and $t$. 
\end{rem}
By the first claim of Lemma \ref{lm:pu}, the claim (\ref{eq:reducedClaim1}) follows if  
we show that
\begin{equation}\label{eq:reduced_claim3}
\left\|\sum_{w\in (T_f^{t})^{-1}(b)}\cL^{t}_{a\to b,w,j}:{\cB}^{r,p}(R)\to {\cB}^{r,p}(R)\right\|_{\ess}\le C_0 \exp((\mu(f) +\varepsilon)t)
\end{equation}
for sufficiently large $t$ and for all $a,b\in A$ and $1\le j\le J(t)$, with a constant $C_0$ independent of $t$, $a$, $b$ and $j$. 

\begin{rem}\label{lem:boundedness}
Letting the lengths  of the  intervals $I^t_j$ in the definition of $\rho^t_j$ be small, we may apply Corollary \ref{cor:bdd} to each component $\cL^{t}_{a\to b,w,j}:{\cB}^{r,p}(R)\to {\cB}^{r,p}(R)$ and see that they are bounded. Consequently the operator $\bL^t$ in  (\ref{eq:bLt_exp}) is bounded.
Further, since we may take the bound on the operator norm of $\bL^t$ locally uniformly in $t$, we obtain (\ref{eq:norm}). 
\end{rem}

\subsection{A preliminary argument for the Proof of Theorem \ref{thm:spectrum}}\label{ss:pre}
As we noted in the beginning of this section, in order to illustrate the main point of the argument clearly, we first prove the conclusion of Theorem \ref{thm:spectrum} under a stronger assumption. For $n\ge 1$ and $\varepsilon>0$, we define $\mathcal{G}'(J_\nu,n,\varepsilon;p)$ as the set of $f\in \Func(y_{\min},y_{\max},\kappa_0)$ such that, for sufficiently large $t>0$ and for any $z=(x,y)\in X_f$, the condition (\ref{eq:mult2}) holds for any $\xi\in [-\theta_0,\theta_0]$ with $\excep=\emptyset$ in the summation. We assume that the roof function $f$ belongs to the set 
\[
\mathcal{G}'=\bigcap_{\nu= 1}^{\nu_0}\bigcap_{m=1}^\infty  
\bigcap_{n\ge 1}\mathcal{G}'(J_{\nu},n, 1/m;p)\subset \Func(y_{\min},y_{\max},\kappa_{0}).
\]
\begin{rem}
From the discussion preceding to Theorem \ref{thm:spectrum}, we  expect that the subset $\mathcal{G}'$ above is also prevalent in $\Func(y_{\min},y_{\max},\kappa_0)$.
The proof of Theorem \ref{thm:spectrum} would be simpler if this was true, as we will see below. But some technical difficulties (related to interference between perturbations) prevent us from this. We therefore resort to a more involved argument presented in the next subsection.  
\end{rem}

We continue the argument in the last subsection  under the assumption as above. We assume that the lengths of the intervals $I_j^t$ in the choice of the functions $\rho^t_j$ are very small. (The precise  
 condition will be given in  Remark \ref{rem:widthI}.) 
Let us take and fix a point\footnote{We will ignore $j$'s such that $\supp \rho_j^t \cap R=\emptyset$.} $z_0=z_0(j)\in  \supp \rho_j^t \cap R$. 
For each $w\in (T_f^{t})^{-1}(b)$ with $R^t_{a\to b,w}\neq \emptyset$,  
let $q=q(w)\in Q$ be the unique point satisfying $\kappa_a(q(w))\in U^t_{b,w}$ and 
$T^t_f(\kappa_a(q(w)))=\kappa_b(z_0)$. 
Then let $S(w)$ and $E(w)\ge 1$ be real numbers such that  
\begin{equation}\label{eq:SEdef}
(DA_{a\to b,w}^t)_{q(w)}=\begin{pmatrix}E(w) &0\\
-S(w)E(w) &1
\end{pmatrix}.
\end{equation}
We divide the set $(T_{f}^{t})^{-1}(b)$ into disjoint subsets $\back_\nu$, $1\le \nu\le \nu_0$, so that 
$w\in (T_{f}^{t})^{-1}(b)$ is contained in $\back_\nu$ only if $E(w)\in [e^{a_\nu t}, e^{b_\nu t}]$.
  Further, letting $t$ be sufficiently large, we may and do suppose that $\back_{\nu}=\emptyset$ if 
\begin{equation}\label{eq:JnuEmpty}
[\chi_{\min}(f), \chi_{\max}(f)]\cap \mathrm{int} J_\nu =\emptyset. 
\end{equation}
Then the operator on the left hand side of (\ref{eq:reduced_claim3}) is expressed as
\[
\sum_{w\in (T_f^{t})^{-1}(b)}\cL^{t}_{a\to b,w,j}=\sum_{\nu=1}^{\nu_0}\Psi_\nu\circ \Phi_\nu
\]
where
\[
\Phi_\nu:\cB^{r,p}(R)\to \bigoplus_{w\in \back_\nu} \cB^{r,p}_{S(w),E(w)}(R),\quad
\Phi_\nu(u)=(\cL^t_{a\to b,w,j}u)_{w\in \back_\nu} 
\]
and
\[
\Psi_\nu:\bigoplus_{w\in  \back_\nu} \cB^{r,p}_{S(w),E(w)}(R)\to \cB^{r,p}(R),\quad
\Psi_\nu((u_w)_{w\in \back_\nu})=\sum_{w\in \back_\nu} u_w. 
\]
Here we suppose that $\bigoplus_{w\in  \back_\nu} \cB^{r,p}_{S(w),E(w)}(R)$ is equipped with the norm
\[
\|(u_w)\|:=\left(\sum_{w\in \back_\nu} \|u_w\|_{r,p,S(w),E(w)}^{2p}\right)^{1/2p}
\]
Then, from Lemma \ref{lm:pu} and Proposition \ref{lm:local1}, the essential operator norm of $\Phi_\nu$ is bounded by
\begin{equation}\label{eq:nonlinear}
C_0 \max_{w\in \back_\nu} E(w)^{1/2p}\le C_0 \exp(b_\nu t/(2p)).
\end{equation}
\begin{rem}\label{rem:widthI}
To get the estimate (\ref{eq:nonlinear}),  we apply Proposition \ref{lm:local1} to each $\cL^{t}_{a\to b,w,j}$. For this purpose, we have to assume that the lengths of intervals $I_j^t$ in the choice of the functions $\rho^t_j$ are sufficiently small. 
This is of course possible. The point is that the constant denoted by $C_0$ in (\ref{eq:nonlinear}) does not depend on the choice of  $\rho_j^t$.
\end{rem}
From Proposition \ref{lm:main} and (\ref{eq:card_backward_orbit}), the essential operator norm of $\Psi_\nu$ is bounded by
\[
 C_0 \exp( (h(f)+\varepsilon)t)(p-1)/2p)\cdot \Delta_\nu^{1/2p}
\]
where $\Delta_\nu$ is the quantity defined in Proposition \ref{lm:main} in the setting 
\[
\{(S(i),E(i))\mid i=1,\cdots, M:=\# \back_\nu\}=\{ (S(w),E(w))\mid w\in \back_\nu\}.
\]
\begin{rem}
To deduce the estimate above, we used  (\ref{eq:card_backward_orbit}) to bound $\# \back_\nu$. Note also that, from the condition (\ref{eq:choice_r}) in the choice of $r$, the latter factor $M^{p-1}\Delta \ge M^{p-1}$ on the right hand side of the inequality (\ref{eq:main_claim}) of Proposition~\ref{lm:main} exceeds the former factor
$M^{2p-1}/(\min_{1\le i\le M} E(i))^{2pr} \lesssim \exp(-2pr\cdot \chi_{\min} t)M^{2p-1}$. 
\end{rem}
Since we are assuming that $f\in \mathcal{G}'$, we have that 
\[
\Delta_\nu \le \exp((\max\{p  h(f)-a_\nu,0\}+p(b_\nu-a_\nu)+\varepsilon)t)
\] 
for sufficiently large $t$, uniformly in $a,b\in A$ and $1\le j\le J(t)$. 
Therefore we conclude that the essential operator norm of $\Psi_\nu\circ \Phi_\nu$ is bounded by
\[
\exp\!\left(\!\frac{b_\nu +(p-1)(h(f)+\varepsilon)+\max\{p h(f)-a_\nu,0\}+p(b_\nu-a_\nu)+2\varepsilon}{2p}\cdot t\!\right)
\]
provided that  $t$ is sufficiently large. 
By the definition of $\mu(f)$ in (\ref{eq:muf}) and arbitrariness of  $\varepsilon>0$, this implies (\ref{eq:reduced_claim3}).  
(Note that $\Psi_\nu\circ \Phi_\nu=0$ if (\ref{eq:JnuEmpty}) holds.)

\subsection{Proof of Theorem \ref{thm:spectrum}}\label{ss:pf_Thspec}
We explain how we modify the argument in the last subsection in order to get the same conclusion under the weaker assumption of Theorem \ref{thm:spectrum}. The idea is not difficult: we use the fact that the exceptional set $\mathcal{E}$ is relatively small as we formulate in (\ref{eq:cardSigma}) below.

We resume the argument in Subsection \ref{ss:reduce}.
Recall that we are considering an arbitrarily small number $\varepsilon>0$. 
Let $m$ and $m'$ be large integers that we will specify in the course of the argument.  We take $n\ge n_0(1/m)$ so large that 
\begin{equation}\label{eq:condn}
\nu_0 p\cdot \lceil 10 m \bar{\chi}_{\max}\rceil \le  \exp(\varepsilon n). 
\end{equation}
In the following we assume that $f$ belongs to 
\[
\bigcap_{\nu=1}^{\nu_0} \mathcal{G}(J_\nu,n,1/m,1/m';p).
\]

We take $t_0>0$ so that the conditions in the definitions of 
$\mathcal{G}(J_\nu,n,1/m,1/m';p)$ for $\nu=1,\cdots, \nu_0$ hold for $t\ge t_0$. 
From the definition of $\mathcal{G}(\cdot)$ in Theorem \ref{thm:multiplicity}, this implies that, for any $t\ge t_0$, any $z=(x,y)\in X_f$ with $x\notin \mathrm{Per}_{1/m'}(\tau,n)$ and $1\le \nu\le \nu_0$, there exists a subset  $\excep=\excep_\nu(z,t;f)\subset \tau^{-n}(x)$ with $\#\excep\le p\lceil 10 m a_\nu \rceil$ such that the condition (\ref{eq:mult2}) holds for any $\xi\in [-\theta_0,\theta_0]$ with $J=J_\nu$. We put 
\[
\excep(z,t;f)=\cup_{\nu=1}^{\nu_0} \excep_\nu(z,t;f).
\]
From the condition (\ref{eq:condn}) in the choice of $n$, we have 
\begin{equation}\label{eq:cardSigma}
\# \excep(z,t;f)\le \exp(\varepsilon n).
\end{equation}
We may and do assume further that  $t_0$ is so large that $t_0>2n\cdot y_{\max}$ and also  
\begin{equation}\label{eq:backnu}
\frac{1}{t}\log |\det DT_f^t(w)|\in 
[\chi_{\min}(f)-\varepsilon, \chi_{\max}(f)+\varepsilon ]
\quad \mbox{ for any $w\in X_f$ and $t\ge t_0$.}\end{equation}

We prove that (\ref{eq:reduced_claim3}) holds for all $a,b\in A$ and $1\le j\le J(t)$ if $t\ge t_0$ is sufficiently large. (Notice that the constant $C_0$ in (\ref{eq:reduced_claim3}) have to be uniform for $a,b$, $t$ and $j$.)
Suppose $t\ge t_0$ and consider arbitrary $a,b\in A$ and $1\le j\le J(t)$. We take a point $z_0=z_0(j)\in \supp\rho_j^t \cap R$ and write $\kappa_{b}(z_0)=(x_0,y_0)$. 
For each point $x\in \tau^{-k}(x_0)$, we define
\[
t(k,x)=f^{(kn)}(x)+y_0
\]
so that $T^{t(k,x)}(x,0)=(x_0,y_0)=\kappa_b(z_0)$ for $x\in \tau^{-kn}(x_0)$. 
Then, we construct the subsets $H_k\subset \tau^{-kn}(x_0)$ for $k\ge 0$ inductively as follows. 
For $k=0$, we set $H_0=\{x_0\}$. 
If $H_{k-1}$ for $k\ge 1$ has been defined, let $H_k$ be the set of points  $x\in \tau^{-kn}(x_0)$ satisfying  
\begin{itemize}
\setlength{\itemsep}{3pt}
\item[(H1)] $x':=\tau^{n}(x)\in \tau^{-(k-1)n}(x_0)$ belongs to $H_{k-1}$,
\item[(H2)] $t-t(k-1,x')>t_0$, and
\item[(H3)]  (a)  $x'\in \mathrm{Per}_{1/m'}(\tau,n)$, or else 
(b)  $x\in 
\excep((x',0),t-t(k-1,x');f)$. 
\end{itemize}
\begin{rem} The condition (H2) ensures that the subset $\excep((x',0),t-t(k-1,x);f)$ in the condition (H3) is well-defined. 
\end{rem}

We check that the number of points in $H_k\subset \tau^{-kn}(x_0)$ is relatively small compared with $\#\tau^{-kn}(x_0)=\ell^{kn}$. Let us say that $x\in H_{k+\nu}$ is a descendant of $\nu$-th generation of $x'\in H_{k}$ if $\tau^{\nu n}(x)=x'$. Observe that
\begin{itemize}
\item[(1)] if $x'\notin \mathrm{Per}_{1/m'}(\tau,n)$, the number of its descendants of the first generation is bounded by $\exp(\varepsilon n )$ from  (\ref{eq:cardSigma}), and 
\item[(2)] if $x'\in \mathrm{Per}_{1/m'}(\tau,n)$, the number of its descendants of the first generation is $\ell^{n}$. 
\end{itemize}
In the case (2) above, the bound on the number of descendants is not effective. 
But, if the case (2) happens for $x'$, the same will not happen for most of its descendants for several generations. More precisely, for arbitrarily large $\nu_0>0$, we may let $m'$ be so large (depending on $n$) that the descendant of $x'$ of $\nu$-th generation with $\nu\le \nu_0$ is not contained in $\mathrm{Per}_{1/m'}(\tau,n)$ but for at most one exception. Therefore,  letting $m'$ be large, we may suppose 
\begin{equation}\label{eq:sk}
\#H_k\le C_0 \ell^n \exp(2\varepsilon k n )\quad \mbox{for $k\ge 0$}
\end{equation}
where $C_0>0$ is a constant depending only on $\varepsilon$.

Let $\mathcal{H}$ be the set of pairs $(k,x)$ of an integer $k\ge 0$ and a point $x\in H_k$.
We say that a pair $(k,x)\in \mathcal{H}$ is \emph{terminal} if 
$t-t(k,x)\le t_0$ and write $\mathcal{H}_{\mathrm{term}}\subset \mathcal{H}$ for the set of such pairs. If a pair $(k,x)\in \mathcal{H}$ is terminal, there is no descendant of $x\in H_k$.

Using the definitions prepared above, we divide the set $(T^t_f)^{-1}(b)$ into several (disjoint) subsets. For each $w\in (T^t_f)^{-1}(b)$ with $R^t_{a\to b,w}\neq \emptyset$, let $q(w)\in R^t_{a\to b, w}$ be the point such that 
\[
\tilde{q}(w):=\kappa_a(q(w))\in U^t_{b,w}\quad \mbox{and}\quad  
T^t_{f}(\tilde{q}(w))=\kappa_b(z_0). 
\]
For each $(k,x)\in \mathcal{H}$,  let $Q(k,x)$ be the set of points  $w\in (T_{f}^{t})^{-1}(b)$ with $R^t_{a\to b,w}\neq \emptyset$ such that 
\begin{equation}\label{eq:tf}
T_f^{s_{kn}(\tilde{z}_0,\tilde{q}(w);t)}(\tilde{q}(w))=(x,0)\quad\mbox{but}\quad 
T_f^{s_{(k+1)n}(\tilde{z}_0,\tilde{q}(w);t)}(\tilde{q}(w))\notin H_{k+1}\times\{0\}
\end{equation}
where $\tilde{z}_0=(x_0,y_0)=\kappa_b(z_0)$. (Recall (\ref{eq:sn}) for the definition of $s_{k}(z,w;t)$.)
Clearly the set $(T^t_f)^{-1}(b)$ splits into the disjoint subsets $Q(k,x)$ for $(k,x)\in \mathcal{H}$, provided that we ignore $w\in (T_{f}^{t})^{-1}(b)$ with $R^t_{a\to b,w}= \emptyset$.
\begin{rem}
The former condition in (\ref{eq:tf}) implies  that 
\[
s_{(k-1)n}(\tilde{z}_0,\tilde{q}(w);t)=t-t(k-1,\tau^{n}(x))>t_0>2n\cdot y_{\max}
\]
and hence that  $s_{(k+1)n}(\tilde{z}_0,\tilde{q}(w);t)$ in the latter condition is well-defined. 
\end{rem}

In the case where a pair $(k,x)\in \mathcal{H}$ is terminal, we have $t-t(k,x)\le t_0$ by definition and we have $T_f^{t-t(k,x)}(\tilde{q}(w))=(x,0)$.
In particular, we have 
\begin{equation}\label{eq:cardkx}
\# Q(k,x)\le \ell^{t_0/y_{\min}}\qquad \mbox{if }(k,x)\in \mathcal{H}_{\mathrm{term}} .
\end{equation}

In the case where a pair $(k,x)\in \mathcal{H}$ is \emph{not} terminal, we decompose the subset  $Q(k,x)\subset (T_{f}^{t})^{-1}(b)$ further. In this case, we have $t-t(k,x)>t_0$. Further, from the definition of $H_k$'s, we have, for $w\in Q(k,x)$, that
\[
T_f^{s_{kn}(\tilde{z}_0,\tilde{q}(w);t )}(\tilde{q}(w))=T_f^{t-t(k,x)}(\tilde{q}(w))=(x,0)\notin \mathrm{Per}_{1/m'}(\tau,n)
\]
and
\[
T_f^{s_{(k+1)n}(\tilde{z}_0,\tilde{q}(w);t )}(\tilde{q}(w))=(\tilde{x},0)\quad \mbox{with $\tilde{x}\notin \excep((x,0),t-t(k,x);f)$.}
\]
From (\ref{eq:backnu}), we can divide $Q(k,x)$ into disjoint subsets $Q_{\nu}(k,x)$ for $1\le \nu\le \nu_0$ so that $w\in Q(k,x)$ belongs to $Q_{\nu}(k,x)$ only if 
\begin{equation}\label{eq:eQ}
\det (DT_f^{t-t(k,x)}(\tilde{q}(w)))\in [e^{a_\nu (t-t(k,x))}, e^{b_\nu (t-t(k,x))}]
\end{equation}
and also that $Q_{\nu}(k,x)=\emptyset$ if (\ref{eq:JnuEmpty}) holds. 

We now estimate the essential operator norm of the operator on the left hand side of the claim (\ref{eq:reduced_claim3}). We decompose the operator into several parts, correspondingly to the decomposition of $(T_f^{t})^{-1}(b)$ into $Q_\nu(k,x)$ for $(k,x)\in \mathcal{H}$, and estimates the essential operator norms of those parts. (We ignore  $w\in (T_{f}^{t})^{-1}(b)$ with $R^t_{a\to b,w}= \emptyset$ since $\cL^t_{a\to b,w,j}$  vanishes for such $w$.)

In general, we have 
\[
\left\|\cL^{t}_{a\to b,w,j}:{\cB}^{r,p}(R)\to {\cB}^{r,p}(R)\right\|_{\ess}\le C_0 e^{(\chi_{\max}(f)+\varepsilon)t/2p}
\]
by Proposition \ref{lm:local1} and Proposition \ref{lm:main} (in the trivial case of $M=1$ and $\Delta=1$.) Hence, by a simple estimate using  (\ref{eq:sk}) and (\ref{eq:cardkx}), we obtain
\begin{multline*}\label{eq:term}
\left\|\sum_{(k,x)\in \mathcal{H}_{\mathrm{term}}}\sum_{w\in Q(k,x)}\cL^{t}_{a\to b,w,j}:{\cB}^{r,p}(R)\to {\cB}^{r,p}(R)\right\|_{\ess}
\\
\le C_0 
\left(
\sum_{k\le t/(n y_{\min})} \ell^{t_0/y_{\min}}\cdot \ell^n \exp(2\varepsilon kn)
\right)\cdot 
e^{(\chi_{\max}(f)+\varepsilon)t/2p}
\end{multline*}
where the range of $k$ in the sum on the right hand side is restricted to $k\le t/(n y_{\min})$ because 
 $H_k$ is empty if $nk\cdot y_{\min}> t$. Since we have
\begin{equation}\label{eq:muf2}
\mu(f)\ge \frac{(p-1)h(f)+\chi_{\max}}{2p}>\frac{\chi_{\max}}{2p}
\end{equation}
from the definition and since we may suppose that $\varepsilon>0$ is small, we see that the right hand side of the inequality above above is bounded by $e^{(\mu(f)+\varepsilon) t}$ when  $t$ is sufficient large. 

We next consider $(k,x)\in \mathcal{H}$ which is \emph{not} terminal.
First of all, observe that, for the case of $(0,x_0)\in H_0$, the argument in the last subsection applies to  
\[
\sum_{w\in Q(0,x_0)} \cL^{t}_{a\to b,w,j}:{\cB}^{r,p}(R)\to {\cB}^{r,p}(R)
\]
 and we can conclude that the essential operator norm of this operator is bounded by $C_0 \exp((\mu +\varepsilon)t)$. (Note that the subset $Q(0,x_0)$ by definition does not contain the problematic elements $w\in (T^t_f)^{-1}(b)$ such that 
 $T^{s_n(\tilde{z}_0,\tilde{q}(w);t)}_f(\tilde{q}(w))\in \mathcal{E}(\tilde{z}_0,t;f)$.)
 Below we see that a similar argument applies to the case $k>0$.
  
Suppose that $(k,x)\in \mathcal{H}$ is not terminal and $w\in Q(k,x)$. 
We consider the local chart $\kappa_c:Q\to X_f$ for $c=(x,0)\in X_f$ so that the point $(x,6\delta_0)$ belongs to $\kappa_{c}(R)$. Let $V\subset Q$ be the neighborhood of $(0,6\delta_0)$ that is mapped by $T_f^{t(k,x)-6\delta_0}\circ \kappa_c$ bijectively onto $\kappa_b(R)$.  
We define 
\begin{align*}
E(w)&:=E(\tilde{q}(w),t-t(k,x);f):=\det (DT_f^{t-t(k,x)}(\tilde{q}(w)))\ge 1
\intertext{and} 
S(w)&:=-F(\tilde{q}(w),t-t(k,x);f)/E(\tilde{q}(w),t-t(k,x);f)
\end{align*}
where $E(\cdot)$ and $F(\cdot)$ on the right hand sides are those in (\ref{eq:EF}), so that
\begin{equation}\label{eq:SEkxdef}
(DA_{a\to c}^{t-t(k,x)+6\delta_0})_{q(w)}=\begin{pmatrix}E(w) &0\\
-S(w)E(w) &1
\end{pmatrix}
\end{equation}
where $A^t_{a\to c}:R^t_{a\to c}\to R$ is defined by (\ref{def:Aab}) and (\ref{def:Rab}) with $b$ replaced by $c$. 
Then we may express the operator 
\[
\sum_{w\in Q(k,x)} \cL^{t}_{a\to b,w,j}:{\cB}^{r,p}(R)\to {\cB}^{r,p}(R)
\]
as  
\[
\sum_{w\in Q(k,x)}\cL^{t}_{a\to b,w,j}=\Xi_{k,x}\circ\left(
\sum_{1\le \nu\le \nu_0}  \Psi_{k,x,\nu}\circ \Phi_{k,x,\nu}\right) 
\]
where the operators $\Xi_{k,x}$, $\Psi_{k,x,\nu}$ and $\Phi_{k,x,\nu}$ are defined as follows: The operators
\[
\Phi_{k,x,\nu}:\cB^{r,p}(R)\to \bigoplus_{w\in Q_\nu(k,x)} \cB^{r,p}_{S(w),E(w)}\left(V\right)
\] 
and
\[
\Psi_{k,x,\nu}:\bigoplus_{w\in Q_\nu(k,x)} \cB^{r,p}_{S(w),E(w)}\left(V\right)\to \cB^{r,p}\left(V\right)
\]
are respectively analogues of the operators $\Psi_{\nu}$ and $\Phi_{\nu}$ considered in the last subsection and  defined precisely by
\begin{align*}
&\Phi_{k,x,\nu}(u)=\left((\rho^t_{a\to b,w,j}\cdot u)\circ (A^{t-t(k,x)+6\delta_0}_{a\to c}|_{R^t_{a\to b,w}})^{-1}\right)_{w\in Q_\nu(k,x)} 
\intertext{
and}
&\Psi_{k,x,\nu}\left((u_w)_{w\in Q_\nu(k,x)}\right)=\sum_{w\in Q_\nu(k,x)} u_w.
\end{align*}
On the other hand, we define   
\[
\Xi_{k,x}:\cB^{r,p}(V)\to \cB^{r,p}(R),\quad \Xi_{k,x}u=u\circ  (A^{t(k,x)-6\delta_0}_{c\to b}|_{V})^{-1}.
\]
For the operator $
\sum_{1\le \nu\le \nu_0}  \Psi_{k,x,\nu}\circ \Phi_{k,x,\nu}$, 
the situation is parallel to that considered in the last subsection and hence  we can get the estimate
\[
\left\|\sum_{1\le \nu\le \nu_0}  \Psi_{k,x,\nu}\circ \Phi_{k,x,\nu}:\cB^{r,p}(R)\to  \cB^{r,p}\left(V\right)\right\|_{\ess}
\le C_0\exp((\mu(f)+\varepsilon)(t(k,x)+6\delta_0))
\]
applying  Proposition~\ref{lm:local1}, Lemma \ref{lm:pu} and Proposition \ref{lm:main}. 
For the operator $\Xi_{k,x}$,  we have the estimate
\[
\|\Xi_{k,x}\|_{\ess}\le C_0 
 \exp((\chi_{\max}(f)+\varepsilon)(t-t(k,x)-6\delta_0)/2p)
\]
from Proposition~\ref{lm:local1} and Proposition \ref{lm:main} (in the trivial case of $M=1$ and $\Delta=1$).
Hence, noting (\ref{eq:muf2}), we obtain 
\[
\left\|\sum_{w\in Q(k,x)}\cL^{t}_{a\to b,w,j}:\cB^{r,p}(R)\to \cB^{r,p}(R)\right\|_{\ess}
\le C_0\exp((\mu(f)+\varepsilon)t)
\]
provided that $\varepsilon>0$ is sufficiently small. 
Therefore we conclude  (\ref{eq:reduced_claim3}) by summing these estimates for $(k,x)\in \mathcal{H}\setminus \mathcal{H}_{\mathrm{term}}$ and using  (\ref{eq:sk}) and arbitrariness of $\varepsilon>0$.  

We have proved that the conclusion of Theorem \ref{thm:spectrum} holds for $f\in \mathcal{G}$. 
But notice that, to get the conclusion of Theorem \ref{thm:spectrum} for some $\eta>0$, it is actually enough to show the estimate (\ref{eq:reduced_claim3}) for $\varepsilon=\eta/2$ and some (very) large $t$ according to $\eta$ and $C_0$. Hence the conclusion remains true for small perturbations of $f$. This completes the proof of Theorem \ref{thm:spectrum}.

\section{Proof of Theorem \ref{th:generalized_huber}}\label{sec:pf_th_huber}
In this section, we justify the heuristic argument in Subsection \ref{ss:heu} and relate the distribution of the periods of prime periodic orbits of $T_f^t$ and  spectral properties of the transfer operators $\cL^t$. In principle, we follow the idea presented in the paper \cite{BT08} where a similar statement for hyperbolic diffeomorphisms is proved.

\subsection{A lift of the operator $\bL^{t}$}
We first introduce a kind of lift (or extension) $\mathbb{L}^{t}$ of the transfer operator $\cL^t$.  
Let us recall the definitions of the Banach spaces $\cB^{r,p}(\real^2)$  and $\cB^{r,p}(X_f)$  in Section~\ref{sec:defcBonR} and \ref{sec:pf1} respectively. Using the notation appeared in those definitions, we define the operators
\[
\mathbb{I}:\bigoplus_{a\in A}C^{\infty}(R)\to\bigoplus_{a\in A}\prod_{m,n}\mathcal{S}(\real^{2}),
\qquad 
\mathbb{I}^{*}:\bigoplus_{a\in A}\bigoplus_{m,n}\mathcal{S}(\real^{2})\to\bigoplus_{a\in A}C^{\infty}(\real^2)
\]
by  
\begin{align*}
&\mathbb{I}\left((u_{a})_{a\in A}\right)=\left(\hat{\chi}_{m,n}(u_{a})\right)_{a\in A,m\in\integer_{+},n\in\integer}
\intertext{
and}
&\mathbb{I}^{*}\left((u_{a,m,n})_{a\in A,m\in\integer_{+},n\in\integer}\right)\mapsto\left(u_{a}:=\sum_{m,n}\hat{\chi}'_{m,n}(u_{a,m,n})\right)_{a\in A}
\end{align*}
where $\chi'_{m,n}$ is the function introduced in the proof of Lemma \ref{lem:xnm_trace} and the ranges of the variables $m$ and $n$ are $\mathbb{Z}_{\ge 0}$ and $\mathbb{Z}$ respectively. 
We have  $\mathbb{I}^{*}\circ\mathbb{I}=\mathrm{Id}$
because $\{\chi_{m,n}\}$ is a partition of unity on $\real^2$ and $\chi'_{m,n}\cdot \chi_{m,n}=\chi_{m,n}$. 

Let $\mathbb{B}^{r,p}$ be the Banach space obtained as  the completion of the space $\bigoplus_{m,n}\mathcal{S}(\real^{2})$
with respect to the norm 
\[
\|(u_{m,n})\|_{r,p}^{(\varepsilon_0)}=\left(\sum_{m,n}2^{2rpm}\cdot  \varepsilon(m)\cdot\|u_{m,n}\|_{L^{2p}}^{2p}\right)^{1/{2p}}
\]
where  
\[
\varepsilon(m)=\begin{cases}
\varepsilon_0,&\quad \mbox{if $m=0$;}\\
1,&\quad \mbox{otherwise}
\end{cases}
\]
with $\varepsilon_0>0$ a small constant that we will specify later. From the definition of the norm on $\cB^{r,p}(\real^2)$, we see that the operators $\mathbb{I}$ and $\mathbb{I}^*$ extend to bounded operators
\[
\mathbb{I}:\bigoplus_{a\in A}\cB^{r,p}(R)\to \bigoplus_{a\in A}\mathbb{B}^{r,p},\qquad 
\mathbb{I}^*:\bigoplus_{a\in A}\mathbb{B}^{r,p}\to\bigoplus_{a\in A}\cB^{r,p}(R).
\]
\begin{rem}
In the definition of the norm $\|\cdot\|_{r,p}^{(\varepsilon_0)}$ above, we put the factor $\varepsilon(m)$ by a technical reason. Of course,  
the Banach space $\mathbb{B}^{r,p}$ (as a set) does not depend on the choice of the constant~$\varepsilon_0>0$. If $\varepsilon_0=1$, the operator $\mathbb{I}$ is an isometric injection by definition. 
\end{rem}

For $t\ge 6\delta_0$, we define 
\[
\mathbb{L}^{t}:\bigoplus_{a\in A}\mathbb{B}^{r,p}
\to \bigoplus_{a\in A}\mathbb{B}^{r,p},\quad\mathbb{L}^{t}=\mathbb{I}\circ\bL^{t}\circ\mathbb{I}^{*}.
\]
Also, for a $C^\infty$ function $\varphi:\real\to \real$ compactly supported on $[6\delta_0,+\infty)$, we define
\[
\mathbb{L}^{\varphi}=\int \varphi(t)\cdot \mathbb{L}^{t} dt:\bigoplus_{a\in A}\mathbb{B}^{r,p}\to \bigoplus_{a\in A}\mathbb{B}^{r,p}.
\]

Note that the following diagram commutes:
\[
\begin{CD}
\bigoplus_{a\in A}\mathbb{B}^{r,p}@>{\mathbb{L}^{t}}>>\bigoplus_{a\in A}\mathbb{B}^{r,p}\\
@A{\mathbb{I}}AA @A{\mathbb{I}}AA\\
\bigoplus_{a\in A}\cB^{r,p}(R)@>{\bL^t}>> \bigoplus_{a\in A}\cB^{r,p}(R)\\
@V{\Pi}VV @V{\Pi}VV\\
\cB^{r,p}(X_f)@>>{\cL^t}>\cB^{r,p}(X_f)
\end{CD}
\]
The essential spectral radii of the three operators $\cL^t$, $
\bL^t$ and $\mathbb{L}^t$ in the commutative diagram above    
are same and their peripheral eigenvalues outside of it coincide. Indeed, we have checked this relation between $\cL^t$ and $\bL^t$ in Subsection \ref{ss:defcB} and, similarly,  can check this relation between $
\bL^t$ and $\mathbb{L}^t$.

We define the flat trace of operators on $\bigoplus_{a\in A}\mathbb{B}^{r,p}$. 
Let  $\mathbb{B}_{a,m,n}^{r,p}$ be the Banach space
$L^{2p}(\real^{2})$ equipped with the norm $
\|u\|_{m}^{(\varepsilon_0)}=
2^{rm}\cdot \varepsilon(m)\cdot \|u\|_{L^{2p}}$.
Then a bounded operator $\mathbb{M}:\bigoplus_{a\in A}\mathbb{B}^{r,p}\to \bigoplus_{a\in A}\mathbb{B}^{r,p}$ may be regarded as a matrix of operators whose components are
\[
\mathbb{M}_{(a,m,n)\to(a',m',n')}:\mathbb{B}_{a,m,n}^{r,p}\to\mathbb{B}_{a',m',n'}^{r,p},\quad u\mapsto \pi_{a',m',n'}\circ \mathbb{M}\circ \pi_{a,m,n} u
\]
where $\pi_{a,m,n}$ is the projection to the $(a,m,n)$-component. If all the diagonal components $\mathbb{M}_{(a,m,n)\to(a,m,n)}$ are trace class operators and if 
\[
\sum_{a,m,n}|\Tr\, \mathbb{M}_{(a,m,n)\to(a,m,n)}|<\infty,
\]
we define the flat trace of $\mathbb{M}$ by
\[
\Tr^\flat\,\mathbb{M}=\sum_{a,m,n}\Tr\, \mathbb{M}_{(a,m,n)\to(a,m,n)}.
\]
\begin{lem}
If\/ $\mathbb{M}$ is a trace class operator, its flat trace is well defined and we have
$\Tr^\flat\,\mathbb{M}=\Tr\,\mathbb{M}$.  
\end{lem}
\begin{proof}Suppose that $\mathbb{M}$ is a rank one operator of the form $\mathbb{M}u=y(u)\cdot x$ with 
\[
x=(x_{a,m,n})\in \bigoplus_{a\in A}\mathbb{B}^{r,p}, \quad y=(y_{a,m,n})\in \bigoplus_{a\in A}(\mathbb{B}^{r,p})^*. 
\]
Note that, letting $q>1$ be such that $q^{-1}+(2p)^{-1}=1$, we have  
\[
\|y\|_{(\bigoplus_{a\in A}\mathbb{B}^{r,p})^*}=\left(\sum_{a,m,n}\|y_{a,m,n}\|_{(\mathbb{B}_{a,m,n}^{r,p})^*}^q\right)^{1/q}
\]
where
\[
\|v\|_{(\mathbb{B}_{a,m,n}^{r,p})^*}=2^{-rm}\cdot \varepsilon(m)^{-1}\cdot \|v\|_{L^{q}}.
\]
Hence we have, by H\"older inequality, that 
\begin{align*}
\sum_{a,m,n}|\Tr\,\mathbb{M}_{(a,m,n)\to(a,m,n)}|&\le \sum_{a,m,n}\|x_{a,m,n}\|_{\mathbb{B}_{a,m,n}^{r,p}}\cdot 
\|y_{a,m,n}\|_{(\mathbb{B}_{a,m,n}^{r,p})^*}\\
&\le \|x\|_{\bigoplus_{a\in A}\mathbb{B}^{r,p}}\cdot \|y\|_{\bigoplus_{a\in A}(\mathbb{B}^{r,p})^*}
\end{align*}
and also that
\[
\sum_{a,m,n}\Tr\,\mathbb{M}_{(a,m,n)\to(a,m,n)}= \sum_{a,m,n}y_{a,m,n}(x_{a,m,n}) =y(x).
\]
This and the definition of the trace norm give the conclusion. 
\end{proof}

The flat trace may be defined for operators that are not of the trace class. We introduce such a class of operators.  
We say that a bounded linear operator $\mathbb{M}:\bigoplus_{a\in A}\mathbb{B}^{r,p}\to \bigoplus_{a\in A}\mathbb{B}^{r,p}$ is \emph{triangular} if its components $\mathbb{M}_{(a,m,n)\to(a',m',n')}$ is zero  whenever
\[
2^m\langle n\rangle^2\le 2^{m'}\langle n'\rangle^2.
\]
Obviously, if $\mathbb{M}$ is triangular, all of its diagonal components vanish and hence its flat trace is defined to be zero. Note that the sum and composition of two triangular operators are again triangular.

In the next subsection, we will see that the operator $\mathbb{L}^\varphi$ for a $C^\infty$ function $\varphi:\real\to \real$ compactly supported on $[t_0,+\infty)$ with sufficiently  large $t_0$ is decomposed into a trace class operator and a triangular operator. Hence the flat trace is well defined for $\mathbb{L}^\varphi$. Note that the components of $\mathbb{L}^\varphi$ are written as
\[
\mathbb{L}_{(a,m,n)\to(a',m',n')}^{\varphi}:\mathbb{B}_{a,m,n}^{r,p}\to\mathbb{B}_{a',m',n'}^{r,p},\;\;u\mapsto \hat{\chi}_{m',n'}\circ \left(\int \varphi(t)\cdot \cL_{a\to a'}^{t}dt\right)\circ \hat{\chi}'_{m,n} u.
\]
This is an integral operator with kernel $K^{\varphi}_{(a,m,n)\to(a',m',n')}(z',z)\in \mathcal{S}(\real^4)$ and hence is a trace class operator. If $(a,m,n)=(a',m',n')$, its trace is calculated as the integration of the kernel on the diagonal:
\[
\Tr \mathbb{L}_{(a,m,n)\to(a,m,n)}^{\varphi}=\int_{\real^2} K^{\varphi}_{(a,m,n)\to(a,m,n)}(z,z) dz.
\]
Since $
\sum_{m,n:2^m \langle n^2\rangle  \le k} \hat{\chi}_{m,n}$ 
converges to the Dirac function $\delta_0$ as $k\to \infty$, we find 
\begin{equation}\label{eq:comp_flattr}
\Tr^{\flat}\mathbb{L}^\varphi=\sum_{a,m,n}\Tr \mathbb{L}_{(a,m,n)\to(a,m,n)}^{\varphi}=\sum_{\gamma\in \Gamma}\sum_{n=1}^{\infty}
\frac{|\gamma|\cdot \varphi(n|\gamma|)}{1-E_\gamma^{-n}}
\end{equation}
by straightforward computation. (Though this computation is not very simple, we ask the readers to check it.) Note that the right hand side is what we expect for $\int \varphi(t) \Tr^\flat \cL^t dt$ from (\ref{eq:ABT}).

\subsection{A decomposition of the lifted operator }
We decompose the operator $\mathbb{L}^{t}$ for $t\ge t_0$ (resp. $\mathbb{L}^\varphi$ for  $\varphi\in C^\infty_0([t_0,+\infty))$ into two parts as
\[
\mathbb{L}^{t}=\mathbb{L}^{t}_{\trace}+\mathbb{L}^{t}_{\tracefree}\quad (\mbox{resp. }\mathbb{L}^{\varphi}=\mathbb{L}^{\varphi}_{\trace}+\mathbb{L}^{\varphi}_{\tracefree})
\]
where $\mathbb{L}^{t}_{\tracefree}$ (resp. $\mathbb{L}^{\varphi}_{\tracefree}$) consists of its components $\mathbb{L}_{(a,m,n)\to(a',m',n')}^{t}$ (reps. $\mathbb{L}_{(a,m,n)\to(a',m',n')}^{\varphi}$) that satisfies the condition
\begin{equation}\label{eq:cond_tracefree}
2^{m'}\langle n'\rangle^2< 2^{m+4}\langle n\rangle^2 \exp(-(\chi_{\min}+\varepsilon) t)
\end{equation}
(resp. the same condition with  $t=\min \supp \varphi\ge t_0$) 
and the operator $\mathbb{L}^{t}_{\tracefree}$ (resp. $\mathbb{L}^{\varphi}_{\tracefree}$) consists of the remaining components.  
Clearly the trace-free part $\mathbb{L}^{t}_{\tracefree}$ (resp. $\mathbb{L}^{\varphi}_{\tracefree}$) is triangular provided that $t_0$ is sufficiently large. 

Below we present two lemmas, whose proofs are deferred to the next subsection.
The constants  $r>0$  and $t_0>0$ are assumed to be sufficiently large.
The first lemma tells that the trace-free part is strongly contracting. This is a consequence of the choice of the weights in the definition of the Banach space $\mathbb{B}^{r,p}$ and may be rather obvious.  
\begin{lem} \label{lm:T1}
For $t_0\le t\le 2t_0$, we have 
\[
\left\|\mathbb{L}_{\tracefree}^{t}:\bigoplus_{a\in A}\mathbb{B}^{r,p}\to \bigoplus_{a\in A}\mathbb{B}^{r,p}\right\|\le \exp(\rho t)
\]
where $\rho=\rho_p(f)$. (See $(\ref{eq:bound_ess})$ for the definition of $\rho_p(f)$.) 
\end{lem}
\begin{rem}
As we will see in the proof, we can  actually prove the statement above for arbitrarily small $\rho$ by letting $r$ and $t_0$ larger. 
\end{rem}
The next lemma proves that the trace class part  $\mathbb{L}^\varphi_{\trace}$ is a trace class operator and also gives an estimates on its trace norm. 
\begin{lem}\label{lm:T2}
For a bounded subset $\mathcal{X}$ in $C^\infty_0([-1,1])$, there exists a constant $C_*=C_*(\mathcal{X})$ such that, if $\varphi$ is supported on $[t_0,2t_0]$ and if 
there exists an affine map $A(t)=\alpha t+\beta$ with $\alpha\in (0,1)$ such that the function 
$\varphi\circ A(t)=\varphi(\alpha t+\beta)$ belongs to  $\mathcal{X}$, then 
\[
\|\mathbb{L}^{\varphi}_{\trace}\|_{\Tr}\le C_* \alpha
\quad{and}\quad
\|\mathbb{L}^t_{\trace}\circ \mathbb{L}^{\varphi}\|_{\Tr}\le C_* \alpha\quad \mbox{for $t_0\le t\le 2t_0$.}
\]
\end{lem}
\begin{rem}
The operator $\mathbb{L}^t_{\trace}$ itself will not be a trace class operator. In the proof of the lemma above, we use the fact that the integration with respect to the variable $t$ (with multiplication by $\varphi(t)$) in the definition of $\mathbb{L}^{\varphi}$ reduces the part of functions that have high frequency in the flow direction (that is, the $(a,m,n)$-components with $|n|$ large).  
\end{rem}

We proceed with the proof of  Theorem~\ref{th:generalized_huber}. Below we consider the situation assumed in  Theorem~\ref{th:generalized_huber}.
We write $\mathbb{\Pi}$ for the spectral projector of $\bbL^t$ for the set of eigenvalues on the outside of the disk $|z|\le e^{(\rho+\varepsilon) t}$ where $\rho=\rho_p(f)$. Note that this spectral projector $\mathbb{\Pi}$ is of finite rank and  does not depend on $t$ provided $t\ge t_0$. By letting $t_0$ be larger if necessary, we assume 
\[
\|\mathbb{\Pi}\circ \bbL^t:\cB^{r,p}(X_f)\to \cB^{r,p}(X_f)\|\le \exp\left((\rho+\varepsilon) t\right)\quad \mbox{for }t\ge t_0.
\] 
The next proposition is the key to the proof of Theorem~\ref{th:generalized_huber}.
\begin{prop}\label{lm:ftrace}
Suppose that $\mathcal{X}$ is a bounded subset in $C^\infty_0([-1,1])$  and that $\varphi$ is a $C^\infty$ function supported on $[t_0,2t_0]$ such that $\varphi\circ A(t)=\varphi(\alpha t+\beta)$ belongs to  $\mathcal{X}$ for some  affine map $A(t)=\alpha t+\beta$ with $\alpha\in (0,1)$.
Then we have
\begin{equation}\label{eq:trace}
|\Tr^\flat((1-\mathbb{\Pi})\circ \bbL^{\varphi}\circ \bbL^T|\le C'_* \alpha^{-1}\cdot \exp({(\rho+\varepsilon) T})\qquad \mbox{for any $T\ge t_0$}
\end{equation}
where the constant $C'_*=C'_*(\mathcal{X},t_0)$ depends on the bounded subset $\mathcal{X}\subset C^{\infty}([-1,1])$, $t_0$ and $\varepsilon$, but not on $\alpha$ and $\beta$.
\end{prop}
\begin{proof}
Let us write $T\ge t_0$ as a sum $
T=t_1+t_2+\cdots+t_m$ with $t_0\le t_i\le 2t_0$.
Since the operators $\mathbb{\Pi}$, $\bbL^t$ and $\bbL^{\varphi}$ commute and since $1-\mathbb{\Pi}$ is a projection operator, we may write 
\begin{align*}
(1-\mathbb{\Pi})\circ \bbL^{\varphi}\circ\bbL^{T}
&= (1-\mathbb{\Pi})\circ\bbL^{t_m}\circ \bbL^{\varphi}\circ   (1-\mathbb{\Pi})\circ\bbL^{t_{m-1}}\circ\bbL^{t_{m-2}}\circ\cdots \circ \bbL^{t_2}\circ \bbL^{t_1}
\\
&= (\bbL^{t_m}_{\trace}-\mathbb{\Pi}\circ \bbL^{t_m})\circ \bbL^{\varphi}\circ (1-\mathbb{\Pi})\circ\bbL^{t_{m-1}}\circ \cdots \circ \bbL^{t_2}\circ\bbL^{t_1}\\
&\qquad\qquad +\bbL_{\tracefree}^{t_m}\circ [ \bbL^{\varphi}\circ (1-\mathbb{\Pi})\circ \bbL^{t_{m-1}}\circ \cdots \circ \bbL^{t_2}\circ\bbL^{t_1}].
\end{align*}
Applying the same deformation to the operator in the last  bracket $[\cdot]$ and continuing this procedure,  we  express the operator $(1-\mathbb{\Pi})\circ \bbL^{\varphi}\circ \bbL^T$ as the sum of 
\begin{align*}
&\bbL_{\tracefree}^{t_m}\circ \cdots\circ \bbL_{\tracefree}^{t_1}\circ
\bbL_{\tracefree}^{\varphi},\qquad \bbL_{\tracefree}^{t_m}\circ \cdots\circ \bbL_{\tracefree}^{t_1}\circ
\bbL_{\trace}^{\varphi},\\
&\bbL_{\tracefree}^{t_m}\circ\cdots \circ \bbL_{\tracefree}^{t_{2}}\circ  (\bbL^{t_1}_{\trace}-\mathbb{\Pi}\circ \bbL^{t_1})\circ \bbL^{\varphi}
\intertext{and}
&\bbL_{\tracefree}^{t_m}\circ\cdots \circ \bbL_{\tracefree}^{t_{k+1}}\circ  (\bbL^{t_k}_{\trace}-\mathbb{\Pi}\circ \bbL^{t_k})\circ \bbL^{\varphi}\circ (1-\mathbb{\Pi})\circ  \bbL^{t_{k-1}}\circ\cdots \circ \bbL^{t_1}
\end{align*}
for $k=2,\cdots, m$. 
\begin{rem}
At the last stage of the development above, we find the term
\begin{multline*}
\bbL_{\tracefree}^{t_m}\circ\cdots \circ \bbL_{\tracefree}^{t_{2}}\circ \bbL^{\varphi}\circ (1-\mathbb{\Pi})\circ \bbL^{t_1} \\
=
\bbL_{\tracefree}^{t_m}\circ\cdots \circ \bbL_{\tracefree}^{t_{2}}\circ \bbL^{t_1}\circ (1-\mathbb{\Pi}) \circ \bbL^{\varphi}.
\end{multline*}
This term is decomposed into the first three terms above.
\end{rem}
Notice that the first operator above is triangular and hence its flat trace vanishes. 
From Lemma \ref{lm:T1} and Lemma \ref{lm:T2}, 
the second operator is a trace class operator and its trace norm is bounded by 
\[
 \exp({\rho}(t_1+\cdots+t_m))\cdot C_* \alpha^{-1}.
\]
Similarly the trace norm of the other operators are bounded by  
\[
\exp\left({\rho} \sum_{i=k+1}^m t_i\right)
\cdot \left(C_* \alpha^{-1}+C\right)\cdot 
C\exp\left(({\rho}+\varepsilon) \sum_{i=1}^{k-1} t_i\right) 
\]
for $k=1$ and $k=2,\cdots,m$ respectively, where $C$ is a constant depending only on the choice of $t_0$ and the rank of $\mathbb{\Pi}$.  Hence the claim (\ref{eq:trace}) follows from  the estimates above. 
\end{proof}
Since the spectral projector $\mathbb{\Pi}$ is of finite rank, so is $\mathbb{\Pi}\circ \mathbb{L}^t$ and therefore we have
\begin{equation}\label{eq:Trflat}
\Tr^\flat\left(\mathbb{\Pi}\circ \mathbb{L}^{\varphi}\circ \mathbb{L}^t \right)=
\Tr\left(\mathbb{\Pi}\circ \mathbb{L}^{\varphi}\circ \mathbb{L}^t \right)=
\sum_{i=1}^{k'}\int \varphi(t) e^{\mu_i t} dt. 
\end{equation}
Thus Proposition \ref{lm:ftrace} is convincing. 
Below we give an argument to finish the proof. We first justify the relation (\ref{eq:principal}). 
Let us put 
\[
\tilde{\pi}(T)=
\sum_{\gamma\in\Gamma}\sum_{n=1}^{\lfloor T/|\gamma|\rfloor}\frac{1}{n}\cdot \frac{1}{1-E_{\gamma}^{-n}}.
\]  
Note that, if we defined $\Tr^\flat \cL^t$ by the formula (\ref{eq:ABT}), we might write it formally as 
\[
\tilde{\pi}(T)=\int_{+0}^T \frac{1}{t}\cdot \Tr^{\flat} \cL^t dt.
\] 
The difference between $\pi(T)$ and $\tilde{\pi}(T)$ is bounded by 
\begin{align*}
\sum_{\gamma\in\Gamma}\sum_{n=2}^{\lfloor T/|\gamma|\rfloor}\frac{n^{-1}}{1-E_{\gamma}^{-n}}+
\sum_{\gamma\in\Gamma:|\gamma|\le T}\frac{E_{\gamma}^{-1}}{1-E_{\gamma}^{-1}}
\le
\sum_{n=2}^{\infty}
\sum_{\gamma\in\Gamma:|\gamma|\le T/n}
1+
\sum_{\gamma\in\Gamma:|\gamma|\le T}2\cdot E_{\gamma}^{-1}.
\end{align*}
As we noted in Subsection \ref{ss:heu} this difference is negligible. By the general argument on the topological pressure of flows (see \cite[Theorem C]{Lu} for instance), we have  
\begin{align*}
&\lim_{T\to \infty}\frac{1}{T}\log \sum_{\gamma\in\Gamma:|\gamma|\le T}1= P_{top}(T_f^t,0)=h(f)
\intertext{and}
&\lim_{T\to \infty}\frac{1}{T}\log \sum_{\gamma\in\Gamma:|\gamma|\le T}E_\gamma^{-1}= P_{top}(T_f^t,-\log \det DT_f^t) =0.
\end{align*}
Hence we have
\[
\lim_{T\to \infty}\frac{1}{T}\log |\tilde{\pi}(T)-\pi(T)|\le \frac{h(f)}{2}.
\]
That is, the difference between $\pi(T)$ and $\tilde{\pi}(T)$ converges to $0$ much faster than the error term in the conclusion of Theorem \ref{th:generalized_huber}. 
Therefore it is enough to prove the statement with $\pi(T)$ replaced by $\tilde{\pi}(T)$.

In order to the last step of the proof, we introduce a few definitions.  
Let 
\[
\mu=(h(f)-{\rho})/2.
\]
For large $T\gg t_0$, we take $C^\infty$ functions 
\begin{align*}
&\varphi_{i}^{T}:\real\to[0,1]\quad\mbox{ for }\lceil t_0\rceil+1 \le i\le \lfloor T\rfloor
\intertext{
and}
&\psi_{i}^{T}:\real\to[0,1]\quad\mbox{ for }0\le i\le k(T):=\lceil \mu T/\log2\rceil
\end{align*}
so that  
\begin{itemize}
\item[(i)]  The supports of $\varphi_{i}^{T}$ and $\psi_{i}^{T}$ are contained respectively in the intervals 
\[
I_{i}=[i-1,i+1]\qquad\mbox{and}\quad J_{i}=[T- 2^{-i},T+2^{-k(T)}].
\]
\item[(ii)] Let $A_{i}:[0,1]\to I_{i}$ and $A'_{i}:[0,1]\to J_{i}$
be the (unique) orientation preserving affine bijections. Then the set of functions 
\[
\{\varphi_{i}^{T}\circ A_{i}\}_{1\le i\le\lfloor T\rfloor}\mbox{ and }\quad\{\psi_{i}^{T}\circ A'_{i}\}_{1\le i\le k(T)}
\]
are contained in a bounded subset $\mathcal{X}\subset C^{\infty}_0([0,1])$ that is independent of~$T$. 
\item[(iii)] If we put  $\Psi_k(t):=\sum_{i=\lceil t_0\rceil+1}^{\lfloor T\rfloor}\varphi_{i}^{T}(t)+\sum_{i=0}^{k}\psi_{i}^{T}(t)$ for $1\le k\le k(T)$, we always have that $0\le \Psi_k(t)\le 1/t$ and that  
\begin{align*}
&\Psi_{k(T)}(t)=\begin{cases}
1/t, & \quad\mbox{for $t\in[t_0+2,T]$;}\\
0, & \quad\mbox{for $t\le t_0$ and $t\ge T+2^{-k(T)}$}
\end{cases}
\intertext{
and, for $0\le k\le k(T)-1$,}
&\Psi_{k}(t)=\begin{cases}
1/t, & \quad\mbox{for $t\in[t_0+2,T-2^{-k}]$;}\\
0, & \quad\mbox{for $t\le t_0$ and $t\ge T$.}
\end{cases}
\end{align*}
\end{itemize}
We complete the proof of Theorem~\ref{th:generalized_huber} (assuming Lemma \ref{lm:T1} and \ref{lm:T2}). From the condition (iii) above and (\ref{eq:comp_flattr}),  we have that 
\[
\Tr^\flat \left(\int \Psi_{k(T)-1}(t) \bbL^{t} dt\right)\le \tilde{\pi}(T)\le 
\Tr^\flat \left(\int \Psi_{k(T)}(t) \bbL^{t} dt\right) +\tilde{\pi}(t_0+2)
\]
Hence the difference $\left|\tilde{\pi}(T)-\Tr^\flat\left(\int_{1}^{T}(1/t) (\mathbb{\Pi}\circ \mathbb{L}^{t}) dt\right)\right|$ is bounded by 
\begin{align*} 
&\int \psi_{k(T)}(t) |\Tr^\flat (\mathbb{\Pi}\circ \bbL^{t})| dt\\
&\qquad +\sum_{i=\lceil t_0\rceil+1}^{\lfloor T\rfloor}
\left|\Tr^\flat ((1-\mathbb{\Pi})\circ \bbL^{\varphi_i^T})\right|
+\sum_{i=0}^{k(T)}
\left| \Tr^\flat ((1-\mathbb{\Pi})\circ \bbL^{\psi_i^T}) \right|
\end{align*}
plus a constant independent of $T$. 
By the estimate (\ref{eq:trace}), we see that 
the second and third terms are bounded by 
\[
\sum_{k=1}^{\lfloor T\rfloor}C'_*\exp
\left(({\rho}+\varepsilon) k\right)\quad \mbox{ and } 
\quad 
C'_* \sum_{k=0}^{k(T)} 
\exp(({\rho}+\varepsilon)  T + k\log 2)
\]
 respectively. Hence their sum is bounded by $C \exp((\rho+\mu+\varepsilon) T)$. 
The first term is bounded by 
$C \exp((h(f)-\mu+\varepsilon)T)$ because 
\[
|\Tr^{\flat}(\mathbb{\Pi}\circ \bbL^t)|\le \mathop{\mathrm{rank}}\mathbb{\Pi} \cdot  \exp((h(f)+\varepsilon) t)
\]
for sufficiently large $t$. 
 Therefore, from the choice of $\mu$, we obtain
\[
\left|\tilde{\pi}(T)-\int_{t_0}^{T}\Tr^\flat (\mathbb{\Pi}\circ \bbL^{t}) dt\right|\le C\exp((h(f)+{\rho}+\varepsilon)T/2).
\]
Clearly the conclusion of Theorem
\ref{th:generalized_huber} follows from this estimate and (\ref{eq:Trflat}).
\begin{rem}\label{rem_bar_rho}
In the last part of the argument above, we find the reason for the choice of
$\mu=(h(f)-\rho)/2$. This also explains why we had 
the average $\bar{\rho}$ in the statement of Theorem
\ref{th:generalized_huber}.
\end{rem}

\subsection{Proof of Lemma \ref{lm:T1} and \ref{lm:T2}}
Lemma \ref{lm:T1} and \ref{lm:T2} follows form elementary estimates on the components of the operators $\mathbb{L}^t$ and $\mathbb{L}^{\varphi}$.  
If we let $t_0>0$ be sufficiently large, we have the following two lemmas.
\begin{lem} \label{lm:bbL}
For any $\nu>0$, there exists a constant $C_{\nu}>0$  such that  
\[
\|\mathbb{L}_{(a,m,n)\to(a',m',n')}^{t}:L^{2p}(R)\to L^{2p}(R)\|\le C_{\nu} \exp((\chi_{\max}+\varepsilon)t/p)\cdot \Delta_1(n,n')^{-\nu}
\]
for any $a,a'\in A$, for any integers $n$, $n'$, $m\ge 0$, $m'\ge 0$ and  for  any $t\ge t_0$, where $\Delta_1(n,n')$ is that defined in 
(\ref{eq:Delta1}).
Further, if 
\[
m'>0, \qquad 2^{m'}\langle n'\rangle^2> 2^{m+4}\langle n\rangle^{2}\cdot \exp(-(\chi_{\min}-\varepsilon)t)\quad \mbox{and }\quad  t_0\le t\le 2t_0,
\]
we have 
\[
\|\mathbb{L}_{(a,m,n)\to(a',m',n')}^{t}:L^{2p}(R)\to L^{2p}(R)\|\le C_{\nu} \cdot \max\{2^m\langle n\rangle^2,2^{m'}\langle n'\rangle^2\}^{-\nu}.
\]
\end{lem}
\begin{proof} The claim is proved by inspecting the kernel of $\mathbb{L}_{(a,m,n)\to(a',m',n')}^{t}$ and using integration by parts. 
We omit the detail of the proof because the argument is parallel to that in the latter part of the proof of  Proposition \ref{lm:local1}.
\end{proof}

\begin{lem}\label{lm:bLvarphi} 
Let $\mathcal{X}\subset 
C^\infty([-1,1])$ be a bounded subset.
For any $\nu>0$, there exists a constant $C_{\nu}(\mathcal{X})$ such that, if $\varphi$ is supported on $[t_0,2t_0]$ and if 
there exists an affine map $A(t)=\alpha t+\beta$ with $\alpha>0$ such that the function 
$\varphi\circ A(t)=\varphi(\alpha t+\beta)$ belongs to  $\mathcal{X}$, then we have
\begin{multline}
\|\mathbb{L}_{(a,m,n)\to(a',m',n')}^{\varphi}:L^{2p}(R)\to L^{2p}(R)\|\\
\le 
C_{\nu}(\mathcal{X})\cdot \alpha\cdot  \langle \alpha |n|^2\rangle^{-\nu} \cdot  \Delta_1(n,n')^{-\nu}.
\end{multline}
\end{lem}
\begin{proof}
We proof is again parallel to that of Proposition \ref{lm:local1}.
We write the integral kernel of the operator 
\[
\mathbb{L}_{(a,m,n)\to(a',m',n')}^{\varphi}=\int \varphi(t) \cdot \mathbb{L}_{(a,m,n)\to(a',m',n')}^{t} dt
\]
explicitly and apply integration by parts. This time, we apply integration by parts also to the integration with respect to the variable $t$. (Note that 
 the mapping $A^t_{a\to a'}$ on local charts satisfies $
A^{t+\tau}_{a\to a'}(x,y)=A^{t}_{a\to a'}(x,y+\tau)$
when  $|\tau|$ is small.)  Then we obtain the factor $\alpha^{-1}\cdot  \langle \alpha^{-1} |n|^2\rangle^{-\nu} $ in addition. 
\end{proof}

From the first claim of Lemma \ref{lm:bbL} and the definition of $\mathbb{L}^t_{\tracefree}$, we obtain Lemma \ref{lm:T1}
 provided that we let the constant  $\varepsilon_0>0$ in the definition of $\|(u_{m,n})\|_{r,p}^{(\varepsilon_0)}$ be sufficiently small and let $t_0$ be sufficiently large.
To prove Lemma \ref{lm:T2}, we note 
\[
\mathbb{L}_{(a,m,n)\to(a',m',n')}^{t}=\hat{\chi}'_{m',n'}\circ \mathbb{L}_{(a,m,n)\to(a',m',n')}^{t}.
\]
Recalling Lemma \ref{lem:xnm_trace}, we see that this implies the estimate 
\begin{multline*}
\|\mathbb{L}_{(a,m,n)\to(a',m',n')}^{t}:L^{2p}(R)\to L^{2p}(R)\|_{\Tr}\\\le 
C_0 2^{m'}\langle n'\rangle^{3}\cdot 
\|\mathbb{L}_{(a,m,n)\to(a',m',n')}^{t}:L^{2p}(R)\to L^{2p}(R)\|
\end{multline*}
and the same estimate with $\mathbb{L}_{(a,m,n)\to(a',m',n')}^{t}$ replaced by $\mathbb{L}_{(a,m,n)\to(a',m',n')}^{\varphi}$. 
Hence the estimates on the operator norms of  $\mathbb{L}_{(a,m,n)\to(a',m',n')}^{t}$ and $\mathbb{L}_{(a,m,n)\to(a',m',n')}^{\varphi}$ in 
 Lemma \ref{lm:bbL} and Lemma \ref{lm:bLvarphi} give the corresponding estimates on the trace norm. 
Finally we evaluate the sum of the trace norms of the components of $\mathbb{L}^\varphi_{\trace}$ and $\mathbb{L}^t_{\trace}\circ \mathbb{L}^\varphi$, by using the estimates thus obtained, and conclude Lemma \ref{lm:T2}. (Though this final step is not completely simple, we omit the detail because the estimates involved are straightforward and crude.)

\section{Proof of Theorem \ref{thm:multiplicity}}
\label{sec:multi}
The proof of Theorem \ref{thm:multiplicity} presented below is basically in the same line as the corresponding argument in the author's previous paper~\cite{Tsujii08}. But we need to modify the argument in some places.
\subsection{Families of roof functions}  
We consider the family of $C^\infty$ functions
\begin{equation}\label{eq:family}
f_{\mathbf{s}}(x)=f(x)+\sum_{k=1}^{K}s_{k}\cdot g_{k}(x)\quad \mbox{with parameter }\mathbf{s}=(s_{1},s_{2},\cdots,s_{K})
\end{equation}
for $f\in \Func(y_{\min},y_{\max},\kappa_{0})\subset C_+^\infty(S^1)$ and 
 $C^\infty$ functions 
\begin{equation}\label{eq:function_gk}
g_{k}:S^{1}\to\real,\quad1\le k\le K.
\end{equation}
The range of parameter will be restricted to 
\[
R(\sigma)=\{\mathbf{s}=(s_{1},s_{2},\cdots,s_{K})\mid\;|s_{k}|\le\sigma\mbox{ for $1\le k\le K$}\}
\]
for some small  $\sigma>0$. The choice of the functions $g_k\in C^\infty_+(S^1)$ and the constant $\sigma>0$ will be given in the course of the argument below. 

We consider an interval $J=[a,b]$, as in the statement of Theorem~\ref{thm:multiplicity}. We suppose $0<\varepsilon<\min\{a,1\}$ and set 
\begin{equation}\label{eq:q}
q=q(\varepsilon):=\left\lceil \frac{10a}{\varepsilon}\right\rceil.
\end{equation}
Below we consider an integer $n\ge 1$ and show that the conclusion of Theorem~\ref{thm:multiplicity} holds when $n$ is sufficiently large according to $\varepsilon$. 

Let $x\in S^1$ and $m\ge 1$. For each $p$-tuple of points in $\tau^{-mn}(x)$,  
\[
\bx=(\bx(i))_{i=1}^{p}\in (\tau^{-mn}(x))^p, 
\] 
we set
\begin{equation}\label{eq:Sbfw}
\hat{S}(\bx,k;f_{\mathbf{s}})= \ell^{-k}\sum_{i=1}^{p}
\frac{d}{dx}f_{\mathbf{s}}^{(k)}(\bx(i)).
\end{equation}
For an array $X=(\bx_1,\cdots,\bx_q)$ of $q$ elements in $(\tau^{-mn}(x))^{p}$, we consider the map 
\[
\Phi_{x,X}:\real^{K}\to\real^{q},\quad\Phi_{x,X}(\mathbf{s})=\left(\hat{S}(\bx_j,mn;f_{\mathbf{s}})\right)_{j=1}^{q}.
\]
This is an affine map and its linear part does not depend on $f=f_{\mathbf{0}}$. 

\begin{defn}
\label{def:indep} We say that an (ordered) array of $q$ elements in $(\tau^{-n}(x))^{p}$, 
\begin{equation}
X=(\bx_{1},\bx_{2},\cdots,\bx_{q})
\label{eq:setw}
\end{equation}
is {\em independent} if there is a component $\bx_{j}(i(j))$ of $\bx_j$ for each $1\le j\le q$
such that $\bx_{j}(i(j))$ does not appear as a component of $\bx_{j'}$ if $j'< j$.\end{defn}
The following claim is proved easily. (We omit the proof.) 
\begin{lem}\label{lm:findingS}
For an array $\mathfrak{X}\subset (\tau^{-n}(x))^p$ of $p$-tuples in $\tau^{-n}(x)$, we set
\[
|\mathfrak{X}|:=\{ x'\in \tau^{-n}(x)\mid \mbox{$x'$ is a component of some $\bx\in \mathfrak{X}$}\}\subset \tau^{-n}(x).
\] 
If  $\#|\mathfrak{X}|>p(q-1)$,  there is an independent array of $q$ elements in~$\mathfrak{X}$. 
\end{lem}
The next lemma explains the motivation for Definition \ref{def:indep}. 
\begin{lem}
\label{lm:family}
There exist $n_0>0$ (depending on $q$ and hence on $\varepsilon$) such that, for any $\delta>0$ and any $n\ge n_0$, we can find a family of smooth functions $g_{k}:S^{1}\to\real$, $1\le k\le K$,
such that the following property holds for the family (\ref{eq:family}):
For any  $x\in S^1\setminus \mathrm{Per}_\delta(\tau,n)$, any $m\ge 1$ and any array $\widetilde{\mathfrak{X}}=(\tilde{\bx}_{1}, \tilde{\bx}_{2}, \cdots, \tilde{\bx}_{1})$ of $q$ elements in $(\tau^{-mn}(x))^{p}$ such that 
\[
\mathfrak{X}:=\left(\bx_{j}:=\big(\tau^{(m-1)n}\tilde{\bx}_j(i)\big)_{i=1}^p\in (\tau^{-n}(x))^{p}\right)_{j=1,\cdots,q}
\]
is independent, we have 
\[
\det D\Phi_{x,\tilde{\mathfrak{X}}}|_{Z}\ge 1
\]
for some $q$-dimensional subspace $Z\subset \real^{K}$.
\end{lem}
\begin{proof} 
Let us consider an arbitrary point $s\in S^1\setminus \mathrm{Per}_\delta(\tau,n)$. For $\rho>0$, let $V_s(\rho)$ be the open \hbox{$\rho$-neighborhood} of $s$ in $S^1$. For $q\in \tau^{-n}(s)$, let $U_{s,q}(\rho)$ be the connected component of $\tau^{-n}(V_s(\rho))$ containing~$q$, so that $\tau^{-n}(V_s(\rho))$ is the disjoint union of $U_{s,q}(\rho)$ for $q\in \tau^{-n}(s)$.
Since $s\notin \mathrm{Per}(\tau,n)$, we have $\tau^{k}(q_0)\neq q_1$ for any distinct $q_0,q_1\in \tau^{-n}(s)$ and any $1\le k\le n$.  
So we can choose $\rho(s)>0$ so small that 
\begin{equation}\label{eq:disjoint}
\tau^{k}(U_{s,q_0}(\rho(s)))\cap U_{s,q_1}(\rho(s))= \emptyset
\end{equation}
for any $q_0,q_1\in \tau^{-n}(s)$ and any $1\le k\le n$. In particular,   
if $x\in \tau^{-n}(V_s(\rho(s))$, we have $\tau^{k}(x)\notin \tau^{-n}(V_s(\rho(s))$ for $1\le k\le n$.

We take functions $g_{s,q}:S^{1}\to\real$ for $q\in \tau^{-n}(s)$  so that $g_{s,q}$ is supported on $U_{s,q}(\rho(s))$ and satisfies
\[
\frac{d}{dx} g_{s,q}(x)=2\ell^n\quad\mbox{on $U_{s,q}(\rho(s)/3)$\quad  and }
\quad 
\left|\frac{d}{dx} g_{s,q}(x)\right|<4\ell^n \quad\mbox{on $S^1$.}
\]
By compactness, we can and do take a finite subset $H\subset S^1$  so that $V_s(\rho(s)/3)$ for $s\in H$ cover $S^1\setminus \mathrm{Per}_\delta(\tau,n)$. Finally we define $g_k$, $1\le k\le K$, as a rearrangement of $g_{s,q}$ for $s\in H$ and  $q\in \tau^{-n}(x)$.  

We check that the conclusion of the lemma holds if we define the functions $g_k$, $1\le k\le K$, as above and if $n$ is sufficiently large. 
Suppose that  $x\in S^1$ and arrays $\tilde{\mathfrak{X}}$ and $\mathfrak{X}$ are given as in the statement of the lemma. 
Since $\mathfrak{X}$ is independent, there is $1\le i(j)\le p$ for $1\le j\le q$ such that $\bx_j(i(j))$ is not a component of $\bx_{j'}$ if $j'<j$.  
We take $s\in S^1$ so that $x\in V_s(\rho(s)/3)$ and select $1\le k(j)\le K$ for $1\le j\le q$ so that $g_{k(j)}$ corresponds to $g_{s,q}$ for $q\in \tau^{-n}(s)$ such that $\bx_j(i(j))\in U_{s,q}(\rho(s)/3)$. 
Let $Z$ be the $q$-dimensional subspace of $\real^{K}$ that contains the $s_{k(j)}$-axis for $1\le j\le q$. 
Observe that $D\Phi_{x,\tilde{\mathfrak{X}}}|_{Z}$ is represented by the $q\times q$ matrix whose $(j,j')$-element is   
\[
M_{j,j'}=\sum_{i=1}^p \sum_{l=0}^{mn-1}\ell^{l-mn}\frac{d}{dx} g_{k(j')}(\tau^{l}(\tilde{\bx}_j(i))).
\]
We regard this matrix as the sum of $M^{(0)}=(M^{(0)}_{j,j'})_{j,j'}$ and $M^{(1)}=(M^{(1)}_{j,j'})_{j,j'}$ with  
\[
M^{(0)}_{j,j'}=\sum_{i=1}^p \sum_{l=(m-1)n}^{mn-1}\ell^{l-mn}\frac{d}{dx} g_{k(j')}(\tau^{l}(\tilde{\bx}_j(i)))
\quad\mbox{and}\quad
M^{(1)}_{j,j'}=M_{j,j'}-M^{(0)}_{j,j'}.
\]
From the disjoint property of the orbits of the supports of $g_{k(j)}$ that follows from (\ref{eq:disjoint}) and from the assumption that $\mathfrak{X}$ is independent, we observe that 
\begin{itemize}
\item[(a)] $M^{(0)}$ is  lower triangular  in the sense that  $M^{(0)}_{j,j'}=0$ if $j'>j$,
\item[(b)] the diagonal components of $M^{(0)}$ are $2k$ for some $1\le k\le p$, while the other components are bounded by $2p$ in absolute value, and
\item[(c)] $M^{(1)}$ is a $q\times q$ matrix whose elements are bounded by $4\ell^{-n}/(1-\ell^{-n})$.
\end{itemize} 
Hence if $n\ge n_0$ for some large $n_0$ depending on $q$ (and $p$, $\ell$), we always have 
\[
\det(D\Phi_{x,\tilde{\mathfrak{X}}}|_{Z})=\det(M^{(0)}+M^{(1)})\ge 1.
\]  
This completes the proof.  
\end{proof}
In the following, we fix the family of functions $g_i$ given in the lemma above.

\subsection{The exceptional set}
In this subsections, we investigate the situation where the roof function $f$ does \emph{not} belong to $\mathcal{G}(J,n,\varepsilon,\delta;p)$ and derive a few consequences. 
So let us suppose that there is an arbitrarily  large $t>0$ and  a point $z_0=(x_0,y_0)\in X_f$ with $x_0\notin \mathrm{Per}_{\delta}(\tau,n)$ and $\xi_0\in [-\theta_0,\theta_0]$ such that, for any subset $\excep\subset\tau^{-n}(x_0)$ with $\#\excep\le pq=p\lceil 10a/\varepsilon\rceil$,
we have
\begin{equation}\sum{}^{*}\;\frac{1}{W^{r}(\bw,t;f)(\xi_0,1)}\ge \exp((\max\{p\cdot h(f)-a,0\}+p(b-a)+\varepsilon)t)\label{eq:mult3}
\end{equation}
where the sum $\sum^{*}$ is taken over $\bw=(\bw(1),\cdots,\bw(p))\in\back(z_0,t;J;f)^{p}$
such that $
T_{f}^{s_{n}(z_0,\bw(i);t)}(\bw(i))\notin \excep\times\{0\}$ for \ensuremath{i=1,2,\cdots,p}. 

We begin with a few basic estimates (which hold in general). 
From the definition of $\back(z_0,t;J;f)$, we have 
\[
e^{at}\le E(w)=\ell^{k(z_0,w;t)}\le e^{bt},\quad\mbox{that is,}\quad
\frac{at}{\log \ell}\le k(z_0,w;t)\le \frac{bt}{\log \ell}
\]
for $w\in \back(z_0,t;J;f)$, where $k(z_0,w;t)$ is that defined in (\ref{eq:sn}). Hence, if we set
\[
m:=\left\lfloor\frac{at}{n\log \ell}\right\rfloor,
\]
we have $mn\le k(z_0,w;t)$ and 
\[
f^{(mn)}(T^{s_{mn}(z_0,w;t)}_f(w))\le t \le \frac{(m+1)n \log \ell}{a}\quad \mbox{for  $w\in \back(z_0,t;J;f)$.}
\]
Note that, for each $x\in \tau^{-mn}(x_0)$, we have
\begin{equation}\label{eq:card_bx}
\#\{w\in \back(z_0,t;J;f)\mid T^{s_{mn}(z_0,w;t)}_f(w)=x\}\le 
\ell^{\lfloor bt/\log \ell\rfloor-mn}\le \ell^{n+1} e^{(b-a)t}.
\end{equation}

For each $\bx=(\bx(i))_{i=1}^p\in (\tau^{-n}(x_0))^p$, let us set  
\[
\Delta^*(\bx)=\sum_{\tilde{\bx}\to \bx}{} \left\langle \ell^{mn}|\xi_0-\hat{S}(\tilde{\bx},mn;f)|\right\rangle^{-r}
\]
where $\hat{S}(\tilde{\bx},mn;f)$ is that defined in (\ref{eq:Sbfw}) (with $\mathbf{s}=0$) and the sum $\sum_{\tilde{\bx}\to \bx}$ is taken over those $\tilde{\bx}=(\tilde{\bx}(i))_{i=1}^p\in (\tau^{-mn}(x_0))^p$ satisfying 
\begin{equation}\label{eq:cond_fmn}
\tau^{(m-1)n}(\tilde{\bx}(i))=\bx(i)\quad\mbox{and}\quad 
f^{(mn)}(\tilde{\bx}(i))\le  \frac{(m+1)n \log \ell}{a}
\quad 
\mbox{for $1\le i\le p$.}
\end{equation}
 We claim that the assumption (\ref{eq:mult3}) implies
\begin{equation}\label{eq:sum_bx}
\sum_{\bx\in (\tau^{-n}(x)\setminus \excep)^p}\Delta^*(\bx)\ge \exp\left(\left(\max\{p\cdot h(f)-a,0\}+\frac{\varepsilon}{2}\right)t\right)
\end{equation}
for any subset $\excep\subset \tau^{-n}(x_0)$ with $\#\excep\le pq$, 
provided that $t$ is sufficiently large. 
To check this claim, let us  consider the quantity 
\[
\Delta(\tilde{\bx})
=\sum_{\tilde{\bw}\to \tilde{\bx}}\;\frac{1}{W^{r}(\tilde{\bw},t;f)(\xi_0,2)}\quad \mbox{for $\tilde{\bx}\in (\tau^{-mn}(x))^p$}
\] 
where the sum $\sum_{\tilde{\bw}\to \tilde{\bx}}$  is taken over $\tilde{\bw}=(\tilde{\bw}(i))_{i=1}^p\in\back(z,t;J;f)^{p}$ such that 
\begin{equation}\label{eq:bxtobw}
\tilde{\bx}(i)=T^{s_{mn}(\tilde{\bw}(i))}(\tilde{\bw}(i))\quad \mbox{for $1\le i\le p$.}
\end{equation}
Then, from (\ref{eq:card_bx}), we have  
\[
\Delta(\tilde{\bx})\le C_0 (\ell^{n+1}\exp((b-a)t))^p\left\langle 
\ell^{mn}|\xi_0-\hat{S}(\tilde{\bx},mn;f)|\right\rangle^{-r}
\]
with $C_0$ a constant depending only on $r$ and $\theta_0$, 
 because
\[
\frac{1}{W^{r}(\tilde{\bw},t;f)(\xi_0,2)}\le C_0 \left\langle 
\ell^{mn}|\xi_0-\hat{S}(\tilde{\bx},mn;f)|\right\rangle^{-r}
\]
for $\tilde{\bw}\in \back(z,t;J;f)^{p}$ satisfying (\ref{eq:bxtobw}). Hence, for $\bx\in (\tau^{-n}(x_0)\setminus \excep)^p$, 
\[
\sum_{\tilde{\bx}\to \bx}  \Delta(\tilde{\bx})
\le 
C_0 (\ell^{n+1}\exp((b-a)t))^p \Delta^*(\bx).
\]
If we take the sum of the left hand side above over $\bx\in (\tau^{-n}(x_0)\setminus \excep)^p$, the total equals  the left hand side of (\ref{eq:mult3}). Therefore we obtain the claim (\ref{eq:sum_bx}) provided that $t$ is sufficiently large.

We next give a consequence of (\ref{eq:sum_bx}). Let us write $\by_k$, $1\le k\le \ell^{pn}$, for  the elements of $(\tau^{-n}(x_0))^p$ and suppose that they are sorted so that
$\Delta^{*}(\by_k)\ge \Delta^{*}(\by_{k'})$ if $k\le k'$.
For $1\le k\le \ell^{pn}$, let 
\[
Y_k=\{ x\in \tau^{-n}(x_0)\mid \mbox{ $x$ is a component of $\by_{k'}$ for some $k'\le k$}\}.
\]
Let $k_*$ be the maximum of $1\le k\le \ell^{pn}$  such that 
$\#Y_k \le pq$. 
Set $\excep=Y_{k*}$ in (\ref{eq:sum_bx}). Then, since $\Delta^*(\bx)\le \Delta^*(\by_{k_*})$ for $\bx\in (\tau^{-n}(x_0)\setminus \excep)^p$, we have that
\[
\ell^{np} \cdot  \Delta^*(\by_{k_*})\ge \exp\left(\left(\max\{p\cdot h(f)-a,0\}+\frac{\varepsilon}{2}\right)t\right).
\]  
This implies
\begin{equation}\label{eq:delta_lower}
\Delta^*(\by_k)\ge \frac{1}{\ell^{np}}\exp((\max\{p\cdot h(f)-a,0\}+(\varepsilon/2))t)\quad \mbox{for $1\le k\le k_*$.}
\end{equation}
Since $\# Y_{k_*}> p(q-1)$, 
 we can choose an independent (ordered) array $(\bx_k)_{k=1}^q$ from $\by_{k}$, $1\le k\le k_*$, 
 by using  Lemma \ref{lm:findingS}.
In conclusion, we found an array $(\bx_k)_{k=1}^q$ of $q$ elements in $(\tau^{-n}(x_0))^p$ that is independent and that  (\ref{eq:delta_lower}) holds with $\by_k$ replaced by $\bx_j$ for $1\le j\le q$. 
 
Finally we reconsider about the choice of  $x_0\in S^1$ and $\xi_0\in [-\theta_0,\theta_0]$. 
Recall that these are given from our assumption that the condition in the definition of $\mathcal{G}(J,n,\varepsilon,\delta;p)$ does not hold for $f$. But, by continuity, it is possible to shift these points a little to so that they belong to some  grids and that the conclusion of the argument above remains true for them (with slight difference in the constants). Precisely, for each $m>0$, we choose a set $P(m)$ of points in $S^{1}\times [-\theta_0,\theta_0]$ such that $\#P(m)\le C_0\ell^{2 (1+\varepsilon) m n}$
and that the $\ell^{- (1+\varepsilon) m n}$-neighborhood of those points cover
$S^{1}\times [-\theta_0,\theta_0]$. Then we can shift the point $(x_0,\xi_0)$ to a nearby point in $P(m)$ so that the conclusion at the end of the last paragraph remains true. 

Let us summarize the argument in this subsection as follows:
\begin{lem} \label{lm:cons}
If $f\in\Func(y_{\min},y_{\max},\kappa_{0})$ does not belong to $\mathcal{G}(J,n,\varepsilon,\delta;p)$, we can find 
\begin{itemize}
\item[(a)] an arbitrarily large integer $m\ge 1$, 
\item[(b)] a point $(x_0,\xi_0)\in P(m)$, 
\item[(c)] an independent array $(\bx_k)_{k=1}^q$ of $q$ elements in $(\tau^{-n}(x_0))^p$,
\end{itemize}
such that 
\[
\sum_{\tilde{\bx}\to \bx_k} \left\langle \ell^{mn}|\xi_0-\hat{S}(\tilde{\bx},mn;f)|\right\rangle^{-r}\ge \exp\left(\left(\max\{p\cdot h(f)-a,0\}+\frac{\varepsilon}{2}\right)\frac{mn\log \ell}{a}\right)
\]
where the sum $\sum_{\tilde{\bx}\to \bx_k}$ is taken over $\tilde{\bx}\in (\tau^{-mn}(x_0))^p$ satisfying (\ref{eq:cond_fmn}) with $\bx=\bx_k$.  
\end{lem}

\subsection{The end of the proof}
Let $\varepsilon,\delta>0$ and $J=[a,b]$ are those given in the statement of Theorem \ref{thm:multiplicity}. We take functions $g_k:S^1\to \real$ for $1\le k\le K$ as in Lemma \ref{lm:family} and consider the families (\ref{eq:family}) for all $f\in \Func(y_{\min},y_{\max},\kappa_{0})$. (The choice of $\sigma>0$ will be given below.) 
For each of such families, we 
prove that  $f_{\bs}$ does \emph{not} belong to $\mathcal{G}(J,n,\varepsilon,\delta;p)$
only when the parameter $\bs\in R(\sigma)$ belongs to a subset with zero Lebesgue measure. This implies that the subset $\mathcal{G}(J,n,\varepsilon,\delta;p)$ is a prevalent subset. \begin{rem}For the last statement, recall Remark \ref{rem:prevalence}. The Lebesgue measure on the finite dimensional subspace of $C^\infty(S^1)$ spanned by $g_k$, $1\le k\le K$, is the transverse measure to $\Func(y_{\min},y_{\max},\kappa_{0})\setminus  \mathcal{G}(J,n,\varepsilon,\delta;p)$.
\end{rem}

Let $\eta>0$ be a small real number that we will specify later. (At least, we suppose that $\eta$ is much smaller than $\varepsilon$.) 
Then let $\sigma>0$ be  so small that 
\[
e^{-\eta}\cdot f(x)\le f_{\bs}(x)\le e^{\eta}\cdot f(x)
\qquad\mbox{and}\qquad
|h(f_{\bs})-h(f)|<\eta
\]
for $\bs\in R(\sigma)$ and $f\in \Func(y_{\min},y_{\max},\kappa_{0})$. For $x_0\in S^1$ and $m\ge 1$, let $\back(x_0,mn)$ be the set of points $x$ in $\tau^{-mn}(x_0)$ satisfying
\begin{equation}\label{eq:fmn_bound}
f^{(mn)}(x) \le  e^{\eta}\cdot \frac{ mn \log \ell}{a}.
\end{equation}
Note that, when $m$ is sufficiently large,  we have  
\[
\#\back(x_0,mn)\le  
\exp\left(h(f) \cdot e^{2\eta}\cdot \frac{mn\log \ell}{a}\right)
\]


For an integer $m\ge 1$, a point $(x_0,\xi_0)\in P(m)$, an array $(\bx_j)_{j=1}^q$ of $q$ elements in $(\tau^{-n}(x_0))^p$ and
an array $(\tilde{\bx}_j)_{k=1}^q$ of $q$ elements in $(\tau^{-mn}(x_0))^p$ such that 
\begin{equation}\label{eq:Bx}
\tau^{(m-1)n}(\tilde{\bx}_j(i))=\bx_j(i)\quad \mbox{for $1\le i\le p$ and $1\le j\le q$,}
\end{equation}
we define the function 
\[
\Xi_m((x_0,\xi_0);(\bx_j)_{j=1}^q; (\tilde{\bx}_j)_{j=1}^q):R(\sigma)\to \real
\]
on the parameter space $R(\sigma)$ by 
\[
\Xi_m((x_0,\xi_0);(\bx_j)_{j=1}^q; (\tilde{\bx}_j)_{j=1}^q)(\bs)=
\prod_{j=1}^q \left\langle \ell^{mn}|\xi_0-\hat{S}(\tilde{\bx}_j,mn,f_{\bs})|\right\rangle^{-r}.
\] 
If the array $(\bx_j)_{j=1}^q$ is independent, we have from the choice of the functions $g_i$ that
\[
\int_{R(\sigma)} \Xi_m((x_0,\xi_0);(\bx_j)_{j=1}^q; (\tilde{\bx}_j)_{j=1}^q)(\bs) d\bs \le C_0 \ell^{-mnq}
\]
for a constant $C_0$ depending only on $r$.
Therefore we have
\begin{multline}\label{eq:int}
\sum{}^{**}
\int_{R(\sigma)} \Xi_m((x_0,\xi_0);(\bx_j)_{j=1}^q; (\tilde{\bx}_j)_{j=1}^q)(\bs) d\bs \\
\le C_0 \ell^{-mnq}\cdot \ell^{2(1+\varepsilon)mn}\cdot  \exp\left(pq\cdot h(f) \cdot e^{2\eta}\cdot \frac{mn\log \ell}{a}\right)
\end{multline}
for sufficiently large  $m$, where the sum $\sum^{**}$ is taken over combinations of 
\begin{itemize}
\item a point $(x_0,\xi_0)\in P(m)$, 
\item an independent array $(\bx_j)_{j=1}^q$ of $q$ elements in $\tau^{-n}(x_0))^p$ and
\item  an array $(\tilde{\bx}_j)_{j=1}^q$ in 
$(\back(x_0,mn))^{p}\subset \tau^{-mn}(x_0))^p$ satisfying (\ref{eq:Bx}). 
\end{itemize}

Let $\mathcal{X}\subset R(\sigma)$ be the set of parameters $\bs\in R(\sigma)$ such that $f_{\bs}$ belongs to $\Func(y_{\min},y_{\max},\kappa_{0})$ and  does not satisfy the condition in the definition of $\mathcal{G}(J,n,\varepsilon,\delta;p)$. 
From the conclusion in the last subsection given  in Lemma \ref{lm:cons}, we see that 
\[
\mathcal{X}\subset \limsup_{m\to \infty} \mathcal{X}_m
\]
where 
$\mathcal{X}_m$ is the set of parameters $\bs\in R(\sigma)$ such that 
\[
\sum{}^{**}\Xi_m((x_0,\xi_0);(\bx_j)_{j=1}^q; (\tilde{\bx}_j)_{j=1}^q)(\bs)
\ge 
\exp\left(q\left(p\cdot e^{-\eta}\cdot h(f)-a+\frac{\varepsilon}{2}\right)\frac{mn\log \ell}{a}\right)
\]
Comparing this with (\ref{eq:int}), we see that the Lebesgue measure of $\mathcal{X}_m$ is bounded by 
\[
C_0 \exp
\left(\left(\frac{2a(1+\varepsilon)}{q}+(e^{2\eta}-e^{-\eta}) p\cdot h(f)- \frac{\varepsilon}{2}\right)
\cdot \frac{qmn\log \ell}{a} \right).
\]
From the choice of $q$ in (\ref{eq:q}), we can take small $\eta>0$ (and also $\sigma>0$ accordingly) so that  this bound
decreases exponentially with respect to $m$. 
Hence Lebesgue measure of $\mathcal{X}$ is zero by Borel-Cantelli lemma.

\bibliographystyle{abbrv}
\bibliography{bib}

\begin{thebibliography}{10}

\bibitem{BT08}
V.~Baladi and M.~Tsujii.
\newblock Dynamical determinants and spectrum for hyperbolic diffeomorphisms.
\newblock In {\em Geometric and probabilistic structures in dynamics}, volume
  469 of {\em Contemp. Math.}, pages 29--68. Amer. Math. Soc., Providence, RI,
  2008.

\bibitem{Buser}
P.~Buser.
\newblock {\em Geometry and spectra of compact {R}iemann surfaces}, volume 106
  of {\em Progress in Mathematics}.
\newblock Birkh\"auser Boston, Inc., Boston, MA, 1992.

\bibitem{GLP}
P.~Giulietti, C.~Liverani, and M.~Pollicott.
\newblock Anosov flows and dynamical zeta functions.
\newblock {\em Ann. of Math. (2)}, 178(2):687--773, 2013.

\bibitem{GGK}
I.~Gohberg, S.~Goldberg, and N.~Krupnik.
\newblock {\em Traces and determinants of linear operators}, volume 116 of {\em
  Operator Theory: Advances and Applications}.
\newblock Birkh\"auser Verlag, Basel, 2000.

\bibitem{MR1161274}
B.~R. Hunt, T.~Sauer, and J.~A. Yorke.
\newblock Prevalence: a translation-invariant ``almost every'' on
  infinite-dimensional spaces.
\newblock {\em Bull. Amer. Math. Soc. (N.S.)}, 27(2):217--238, 1992.

\bibitem{MR1191479}
B.~R. Hunt, T.~Sauer, and J.~A. Yorke.
\newblock Prevalence. {A}n addendum to: ``{P}revalence: a translation-invariant
  `almost every' on infinite-dimensional spaces'' [{B}ull.\ {A}mer.\ {M}ath.\
  {S}oc.\ ({N}.{S}.)\ {\bf 27} (1992), no.\ 2, 217--238; {MR}1161274
  (93k:28018)].
\newblock {\em Bull. Amer. Math. Soc. (N.S.)}, 28(2):306--307, 1993.

\bibitem{Lu}
Z.-h. Lu.
\newblock Topological pressure of continuous flows without fixed points.
\newblock {\em J. Math. Anal. Appl.}, 311(2):703--714, 2005.

\bibitem{Mane}
R.~Ma{\~n}{\'e}.
\newblock {\em Ergodic theory and differentiable dynamics}, volume~8 of {\em
  Ergebnisse der Mathematik und ihrer Grenzgebiete (3) [Results in Mathematics
  and Related Areas (3)]}.
\newblock Springer-Verlag, Berlin, 1987.
\newblock Translated from the Portuguese by Silvio Levy.

\bibitem{Yorke}
W.~Ott and J.~A. Yorke.
\newblock Prevalence.
\newblock {\em Bull. Amer. Math. Soc. (N.S.)}, 42(3):263--290 (electronic),
  2005.

\bibitem{PP}
W.~Parry and M.~Pollicott.
\newblock An analogue of the prime number theorem for closed orbits of {A}xiom
  {A} flows.
\newblock {\em Ann. of Math. (2)}, 118(3):573--591, 1983.

\bibitem{Pollicott99}
M.~Pollicott.
\newblock On the mixing of {A}xiom {A} attracting flows and a conjecture of
  {R}uelle.
\newblock {\em Ergodic Theory Dynam. Systems}, 19(2):535--548, 1999.

\bibitem{PS}
M.~Pollicott and R.~Sharp.
\newblock Exponential error terms for growth functions on negatively curved
  surfaces.
\newblock {\em Amer. J. Math.}, 120(5):1019--1042, 1998.

\bibitem{Ruelle86}
D.~Ruelle.
\newblock Locating resonances for {A}xiom {A} dynamical systems.
\newblock {\em J. Statist. Phys.}, 44(3-4):281--292, 1986.

\bibitem{Stoyanov}
L.~Stoyanov.
\newblock Ruelle transfer operators for contact anosov flows and decay of
  correlations.
\newblock 2013.

\bibitem{Triebel}
H.~Triebel.
\newblock {\em Theory of function spaces}, volume~78 of {\em Monographs in
  Mathematics}.
\newblock Birkh\"auser Verlag, Basel, 1983.

\bibitem{MR1167374}
M.~Tsujii.
\newblock A measure on the space of smooth mappings and dynamical system
  theory.
\newblock {\em J. Math. Soc. Japan}, 44(3):415--425, 1992.

\bibitem{Tsujii08}
M.~Tsujii.
\newblock Decay of correlations in suspension semi-flows of angle-multiplying
  maps.
\newblock {\em Ergodic Theory Dynam. Systems}, 28(1):291--317, 2008.

\bibitem{Tsujii10}
M.~Tsujii.
\newblock Quasi-compactness of transfer operators for contact {A}nosov flows.
\newblock {\em Nonlinearity}, 23(7):1495--1545, 2010.

\bibitem{Tsujii12}
M.~Tsujii.
\newblock Contact {A}nosov flows and the {F}ourier-{B}ros-{I}agolnitzer
  transform.
\newblock {\em Ergodic Theory Dynam. Systems}, 32(6):2083--2118, 2012.

\end{thebibliography}

\end{document}